\documentclass[11pt]{article}
\usepackage[utf8]{inputenc}
\usepackage{amssymb,amsmath,amsthm,graphicx,bm,bbm,hyperref,
tikz-cd,subcaption,enumitem}
\usepackage{rotating}
\usetikzlibrary{patterns}
\usepackage[ruled,vlined,linesnumbered]{algorithm2e}

\newtheorem{theorem}{Theorem}[section] 
\newtheorem{lemma}[theorem]{Lemma} 
\newtheorem{corollary}[theorem]{Corollary}
\newtheorem{remark}[theorem]{Remark}
\newtheorem{definition}[theorem]{Definition}
\newtheorem{example}[theorem]{Example}
\newtheorem{proposition}[theorem]{Proposition}

\newcommand{\mc}{\mathcal}
\newcommand{\R}{\mathbb{R}}
\newcommand{\Z}{\mathbb{Z}}
\newcommand{\N}{\mathbb{N}}
\newcommand{\eps}{\varepsilon}

\DeclareMathOperator{\argmin}{argmin}
\DeclareMathOperator{\dist}{dist}
\DeclareMathOperator{\tr}{Tr}
\DeclareMathOperator{\bd}{bd}
\DeclareMathOperator{\bdc}{bdc}

\allowdisplaybreaks[3]

\title{Efficient space-time discretizations for tracking the boundaries of 
reachable sets}
\author{Janosch Rieger \and Kyria Wawryk}
\date{\today}

\begin{document}

\maketitle

\begin{abstract}
The reachable sets of nonlinear control systems can in general only be numerically 
approximated, and are often very expensive to calculate.
In this paper, we propose an algorithm that tracks only the boundaries of 
the reachable sets and that chooses the temporal and spatial discretizations in a 
non-uniform way to reduce the computational complexity. 
\end{abstract}

\textbf{Mathematics Subject Classification:} 93B03, 65L50.

\textbf{Keywords:} Reachable sets of control systems; Numerical approximation; 
Boundary tracking; Efficient discretization; Euler's method.

\section{Introduction}

The reachable set of a control system represents all possible states 
the system can attain at a given time.
Reachable sets have applications in several different fields, which are often,
but not always, related to safety and security
\cite{Colombo:Lorenz, Gerdts:2013, Parise}.
Since explicit formulas for reachable sets are not usually available, 
we instead must rely on numerical approximation. 
These approximations are expensive to compute as they need to track 
an entire set that evolves according to a multivalued flow. 
In particular, they require a spatial discretization of the set as well 
as a temporal discretization of the dynamics.

\medskip

The numerical methods for the computation of reachable sets are split into those 
which are designed specifically for linear control systems, and those which 
are applicable to nonlinear systems.

The reachable sets of linear systems are convex and behave in 
a rather straightforward way. 
This enables the creation of numerical methods which are theoretically robust
and perform well in practice, taking advantage of the fact that the sets can be 
approximated by relatively simple shapes such as zonotopes 
\cite{Althoff:zonotopes,Girard,Reissig:2022} or ellipsoids \cite{Kurzhanski}, 
or by a finite number of support points \cite{Baier:2007,Guernic:Support}.

The shape and behavior of reachable sets of nonlinear control systems is more 
complex, and so it is unclear how they can be discretized
and evolved in the most efficient way. 
For this reason, several different families of numerical methods have been proposed.

One such family applies the linear methods discussed above to local 
linearizations of nonlinear systems. 
Due to the linearization errors, the performance 
of these methods can depend on the degree of their nonlinearity 
(see \cite{Rungger:2018} for an analysis). 

Another family of algorithms, which are based on optimal control 
\cite{Baier:2013, Riedl:2021}, can significantly improve practical performance at 
the expense of losing guaranteed convergence. 
These methods use solvers for optimal control problems to approximate the reachable 
set at a particular time, and thus eliminate the 
need for spatial discretization at intermediate timesteps at the risk of getting 
caught in local minima and losing a significant part of the set. 

Finally, Runge-Kutta methods for nonlinear differential equations can be 
generalized to reachable set approximation. Typically, first-order methods 
are used, as even second-order convergence has prohibitively strict requirements 
on the control system \cite{Veliov, Veliov:affine}. 
Spatial discretization for these methods generally involves representing the reachable set as a subset of a grid \cite{Beyn:2007, Komarov}.
Since these subsets can become very large, storage and computational cost can be reduced by an intelligent choice of discretization \cite{Rieger:2024}, or by considering only the boundaries of the reachable sets \cite{Rieger:boundary}.

\medskip

In this paper, we aim to combine the benefits of the algorithms proposed
in \cite{Rieger:boundary} and \cite{Rieger:2024} in a new numerical
method for reachable set computation.
The boundary tracking algorithm from \cite{Rieger:boundary}
is essentially a fully discretized set-valued Euler scheme, but it reduces computational complexity by evolving only the 
boundary of the reachable set in time, and evaluating only the boundary 
of the multivalued vector field of the control system.
The adaptive refinement strategy pursued in \cite{Rieger:2024} aims
to identify a space-time discretization with a given a priori error
tolerance that minimizes the computational cost of the set-valued Euler scheme.
To develop an algorithm that possesses both advantages, we proceed as follows.

\medskip

After setting our notation and introducing reachable sets and their Euler 
approximations in Sections \ref{sec:notation} and \ref{sec:reachable:sets},
we first have to overcome two technical obstacles in Section 
\ref{sec:boundary:tracking}:
The boundary tracking algorithm as published in \cite{Rieger:boundary}
stores boundaries of discrete reachable sets in an ad-hoc fashion, without exploring the space of all such boundaries from a theoretical perspective. 
In order to take the algorithm further, we must do so now, and embed it into the framework of \cite{Rieger:grid} for working 
with boundaries of discrete sets.
Furthermore, it turns out that a naive generalization of the algorithm to 
nonuniform meshes fails, and we introduce appropriate modifications.

\medskip

In Section \ref{sec:abstract:scheme}, we develop an abstract scheme
that refines a given space-time discretization until a prescribed 
error tolerance is met.
It is a greedy algorithm in the sense that every refinement step
subdivides the time interval where the quotient of the decrease in 
error and the increase in computational cost is maximized.
In Section \ref{sec:main}, we derive a concrete estimator for the 
computational cost of the boundary tracking algorithm for nonuniform
discretizations, and prove that it meets the requirements of the
abstract refinement scheme.
The resulting algorithm is conceptually similar to the method from \cite{Rieger:2024}.
It alternates between computing approximate boundaries of the reachable sets 
for a given space-time discretization and, based on this knowledge of their geometry, 
refining the mesh. 
We demonstrate that this process converges in finite time.

\medskip

Finally, we present two numerical examples in Section \ref{sec:numerical:examples},
which suggest that our algorithm has the potential to outperform the Euler scheme 
with uniform discretization from \cite{Beyn:2007}, the boundary
tracking algorithm with uniform discretization from \cite{Rieger:boundary}
and the Euler scheme with adaptive discretization from \cite{Rieger:2024}.

\medskip

We consciously break away from the tradition to keep algorithms close to the position 
in the document where they are mentioned.
Instead, we display them grouped together on pages \pageref{Alg:BdryEulrStep}
to \pageref{Alg:IterativeMethod}, which allows the reader to retrace 
how the building blocks of the overall algorithm communicate.

\section{Notation}\label{sec:notation} 

We denote $\N_0:=\{0,1,2,\ldots\}$, $\N_1=\{1,2,3,\ldots\}$,
$\N_2=\{2,3,4,\ldots\}$ and $\Z:=\{\ldots-1,0,1,\ldots\}$.
When $J\subset\R$ is an interval, and when it is clear that $k\in\N$ or $k\in\Z$, 
we will write $k\in J$ instead of $k\in J\cap\N$ or $k\in J\cap\Z$.
In addition to the usual $d$-dimensional vector space $\R^d$, we introduce 
the space
\[\R^{0,d}:=\{(x_0,x_1,\ldots,x_d):x_0,x_1,\ldots,x_d\in\R\}\]
and define the cumulative sum
$\Sigma_+:\R^d\to\R^{0,d}$ by
$\left(\Sigma_+x\right)_k:=\sum_{j=1}^{k}x_j$ for all $k\in[0,d]$.
We also denote
\[\R_{+}^d:=\{x\in\R^d:x_i\ge 0\ \forall\,i\in[1,d]\},\quad
\R_{>0}^d:=\{x\in\R^d:x_i>0\ \forall\,i\in[1,d]\},\]
and we extend this notation to $\R^{0,d}$ in the obvious way.
We equip $\R^d$ with 
\[\|\cdot\|_\infty:\R^d\to\R_+
\quad\text{given by}\quad
\|x\|_\infty:=\max_{i\in[1,d]}|x_i|,\]
and we define the ball with center $x\in\R^d$ and radius $r>0$ by 
\[B_r(x):=\{y\in\R^d:\|x-y\|_\infty\le r\}.\]
The cardinality of a set $M\subset\R^d$ is denoted $\#M$,
and its boundary $\partial M$.

\medskip

Let $\mc{K}(\R^d)$ and $\mc{KC}(\R^d)$ denote the collections of all nonempty 
and compact subsets of $\R^d$, and all nonempty, compact and convex subsets of $\R^d$,
respectively. 
The Hausdorff distance $\dist_H:\mc{K}(\R^d)\times\mc{K}(\R^d)\to\R_+$
is given by
\[
\dist_H(M,\tilde{M})
=\max\{\dist(M,\tilde{M}),\,
\dist(\tilde{M},M)\},
\]
where 
\[\dist(M,\tilde{M})=\sup_{x\in M}\inf_{\tilde{x}\in\tilde{M}}\|x-\tilde{x}\|_\infty.\]
To approximate $\mc{K}(\R^d)$, for every $\rho>0$, we consider the grid
$\Delta_\rho:=\rho\Z^d$ and the collections of digital images
\[S_\rho:=\{M\subset\Delta_\rho:M\ne\emptyset\}
\quad\text{and}\quad 
S_\rho^+:=2^{\Delta_\rho}.\]
By $C_\rho$ we denote the collection of all sets $M\in S_\rho$,
i.e.\ sets with the property that for any $x,z\in M$ there exists
a finite sequence $(y_k)_{k=0}^n$ with $y_0=x$, $y_n=z$ and
$\|y_{k+1}-y_k\|_\infty=\rho$ for all $k\in[0,n-1]$.
These are called chain-connected sets in \cite[Definition 9]{Rieger:boundary}
and $8$-connected sets (in the plane) in \cite[Section 2.2.1]{Klette}.

\medskip

Finally, a set-valued map $F:\R^d\to\mc{K}(\R^d)$ is called $L$-Lipschitz 
with $L\in\R_+$ if
\[\dist_H(F(x),F(y))\le L\|x-y\|_\infty\quad\forall\,x,y\in\R^d,\]
and when $M\subset\R^d$, we denote $F(M):=\cup_{x\in M}F(x)$.

\section{Reachable sets and their approximations}\label{sec:reachable:sets}

In the following, we briefly introduce reachable sets and a common 
approach to their numerical computation.

\subsection{Reachable sets of differential inclusions}

We consider the differential inclusion
\begin{equation}\label{eq:ODI}
\dot{x}(t)\in F(x(t))\ \text{for almost every}\ t\in(0,T),\quad x(0)\in X_0,
\end{equation}
where $X_0\in\mc{K}(\R^d)$ is path-connected, and $F:\R^d\to\mc{KC}(\R^d)$ is $L$-Lipschitz 
and satisfies the uniform bound
\begin{equation}\label{eq:Pbound}
\sup_{x\in\R^d}\sup_{f\in F(x)}\|f\|_\infty\le P.
\end{equation}

We are interested in the following sets associated with inclusion \eqref{eq:ODI}.

\begin{definition}\label{def:reachable}
The solution set and the reachable sets of inclusion \eqref{eq:ODI} are
\begin{align*}
\mc{S}^F(T)&:=\{x(\cdot)\in W^{1,1}([0,T],\R^d):x(\cdot)\ 
\text{solves inclusion \eqref{eq:ODI}}\},\\
\mc{R}^F(t)&:=\{x(t):x(\cdot)\in \mc{S}^F(T)\},\quad t\in[0,T].
\end{align*}
\end{definition}

\subsection{The Euler scheme for uniform grids}\label{sec:Euler:scheme}

We discretize the space $\mc{K}(\R^d)$ in a way that is similar to outer Jordan 
digitization, see \cite[Definition 2.8]{Klette}.
\begin{lemma}\label{projection:properties}
For any numbers $\alpha,\rho>0$  with $\alpha\ge\rho/2$, 
the mapping
\[\pi^\alpha_\rho:\mc{K}(\R^d)\to S_\rho
\quad\text{given by}\quad
\pi^\alpha_\rho(M):=B_\alpha(M)\cap\Delta_\rho\]
is well-defined, and for all path-connected $M\in\mc{K}(\R^d)$,
we have $\pi^\alpha_\rho(M)\in C_\rho$.
\end{lemma}

The Euler map approximates the flow of inclusion \eqref{eq:ODI}.

\begin{definition}\label{def:Euler}
For any given $\alpha,h,\rho>0$, 
we define the set-valued functions $\phi_h:\R^d\to\mc{KC}(\R^d)$
and $\Phi_{h,\rho}^\alpha:\R^d\to S_\rho^+$ by
\begin{align*}
\phi_h(x):=x+hF(x)\quad\text{and}\quad
\Phi_{h,\rho}^\alpha(x):=\pi^\alpha_\rho(\phi_h(x)).
\end{align*}
\end{definition}

The following proposition is nearly identical to \cite[Proposition 12]{Rieger:boundary} and
is essential for the boundary tracking algorithm to be 
discussed in Section \ref{sec:boundary:tracking} to work.

\begin{proposition}\label{Prop:SetsAreConnected}
Let $h,\rho>0$, and define
\begin{equation}\label{alpha}
\alpha(h,\rho):=(1+Lh)\tfrac{\rho}{2}.
\end{equation}
Then for any $M\in C_\rho$ and $\alpha\ge\alpha(h,\rho)$, we have 
$\Phi_{h,\rho}^{\alpha}(M)\in C_\rho$. 
\end{proposition}

The following Euler scheme for approximating the reachable sets from Definition 
\ref{def:reachable} employs a uniform discretization: 
Every timestep has length $h>0$, and every point computed is an element 
of $\Delta_\rho$.

\begin{definition}\label{reachable:sets:uniform}
Let $n\in\N_1$, let $\rho>0$, and let $h:=T/n$. 
Then the set of all Euler approximations to inclusion \eqref{eq:ODI} is
\[\mc{S}_{h,\rho}^F(T):=\{(y_k)_{k=0}^n: 
y_0\in\pi^{\rho/2}_{\rho}(X_0),\
y_k\in\Phi^{\alpha(h,\rho)}_{h,\rho}(y_{k-1})\ \forall\,k\in[1,n]\},\]
and the corresponding approximate reachable sets are
\begin{equation}\label{eq:def:unifapproxreach}
\mc{R}_{h,\rho}^F(k):=\{y_k:(y_j)_{j=0}^n\in\mc{S}_{h,\rho}^F(T)\},\quad k\in[0,n].
\end{equation}
\end{definition}

The discrete solution set $S_{h,\rho}^F(T)$ is only used for theoretical purposes.
In practice, the discrete reachable sets are computed using the recursion
\begin{equation}\label{set:iteration}\begin{aligned}
\mc{R}_{h,\rho}^F(0)
&=\pi^{\rho/2}_{\rho}(X_0)\\
\mc{R}_{h,\rho}^F(k)
&=\Phi^{\alpha(h,\rho)}_{h,\rho}(\mc{R}_{h,\rho}^F(k-1))\ \forall\,k\in[1,n].
\end{aligned}\end{equation}
Since $X_0$ is path-connected, and in view of Lemma \ref{projection:properties} 
and Proposition \ref{Prop:SetsAreConnected}, we have $\mc{R}_{h,\rho}^F(k)\in C_\rho$
for all $k\in[0,n]$.

\section{Boundary tracking for nonuniform grids}\label{sec:boundary:tracking}

In this section we introduce a storage format for boundaries of digital images, 
and summarize how the algorithm from \cite{Rieger:boundary} tracks the
boundaries of the reachable sets from \eqref{eq:def:unifapproxreach}.
Finally, we show how the restriction operator from \cite{Rieger:grid}
helps to overcome technical difficulties that arise when nonuniform 
discretizations are considered.

\subsection{Boundary pairs of digital images}

We restate the definition of boundary layers of digital images from \cite{Rieger:boundary}.
They are related to the so-called distance transform \cite[Section 3.4.2]{Klette}.

\begin{definition}
For every $\rho>0$ and $M\subset\Delta_\rho$, we define 
\begin{align*}
\partial_\rho^0M
&:=\{x\in M:
\exists\,z\in M^c\cap\Delta_\rho\ \text{with}\ \|x-z\|_\infty=\rho\},\\
\partial_\rho^jM&:=\{x\in M^c\cap\Delta_\rho:
\dist(x,\partial_\rho^0M)=j\rho\},\quad j\in\N_1,\\
\partial_\rho^{-j}M&
:=\{x\in M:\dist(x,\partial_\rho^0M)=j\rho\},\quad j\in\N_1.
\end{align*}
\end{definition}

The following theorem summarizes the relevant results from \cite{Rieger:grid}:
\begin{itemize}
\item [i)] Every digital image, i.e.\ every element of $M\in S_\rho$, is uniquely 
characterized by the pair $(\partial^0_\rho M,\partial^1_\rho M)$.
\item [ii)] There exists a well-defined mapping $\partial R_\rho^{\rho'}$
that maps the boundary pair of a set $M\in S_\rho$ to the boundary pair
of a metrically \eqref{R:Hausdorff} and topologically \eqref{R:connected}
similar set $R_\rho^{\rho'}(M)\in S_{\rho'}$ that is discretized at a different 
resolution.
\end{itemize}
For an illustration of the theorem see the first diagram in Figure \ref{fig:diagrams}.

\begin{theorem}
For every $\rho>0$, the mapping 
\begin{equation*}
\tr_\rho: S_\rho\to S_\rho^+\times S_\rho^+,\quad
\tr_\rho(M):=(\partial^0_\rho M,\partial^1_\rho M),
\end{equation*}
is injective, and we write 
$\bd_\rho:=\tr_\rho(S_\rho)$ and $\bdc_\rho:=\tr_\rho(C_\rho)$.

\medskip

For nested grids $\Delta_{\rho'}\subsetneq\Delta_\rho$ with $\rho'\in\rho\N_2$,
the operator $R_\rho^{\rho'}:S_\rho\to S_{\rho'}$ given by 
$R_\rho^{\rho'}(M):=B_{\rho'/2}(M)\cap\Delta_{\rho'}$ satisfies
\begin{align}
&\forall\,M\in S_\rho:\quad
\dist_H(R_\rho^{\rho'}(M),M)\le\rho'/2,&\label{R:Hausdorff}\\
&\forall\,M\in C_\rho:\quad
R_\rho^{\rho'}(M)\in C_{\rho'},\label{R:connected}
\end{align}
and the lifted operator
\begin{equation}\label{def:dR}
\partial R_\rho^{\rho'}:\bd_\rho\to\bd_{\rho'},\quad
\partial R_\rho^{\rho'}
:=\tr_{\rho'}\circ R_\rho^{\rho'}\circ\tr_\rho^{-1}
\end{equation}
is well-defined.
\end{theorem}

The mapping $\partial R_\rho^{\rho'}$ can be implemented efficiently,
see \cite[Algorithm 1]{Rieger:grid}.

\subsection{Boundary tracking for uniform grids}

In view of Proposition \ref{Prop:SetsAreConnected}, the operator
\begin{equation}\label{well-defined:Psi}
\Psi_{h,\rho}:\bdc_\rho\to\bdc_\rho,\quad
\Psi_{h,\rho}(D)
:=\tr_\rho\circ\,\Phi^{\alpha(h,\rho)}_{h,\rho}\circ\tr_\rho^{-1}(D)
\end{equation}
with $\alpha(h,\rho)$ as in \eqref{alpha} is well-defined.
It completes the second diagram in Figure \ref{fig:diagrams}.
The boundary tracking algorithm from \cite{Rieger:boundary} uses $\Psi_{h,\rho}$
to lift recursion \eqref{set:iteration},
which evolves the full reachable sets, to a recursion 
\begin{equation*}\begin{aligned}
\tr_\rho(\mc{R}_{h,\rho}^F(0))
&=\tr_\rho(\pi^{\rho/2}_\rho(X_0)),\\
\tr_\rho(\mc{R}_{h,\rho}^F(k))
&=\Psi_{h,\rho}
(\tr_\rho(\mc{R}_{h,\rho}^F(k-1)))
\quad\forall\,k\in[1,n].
\end{aligned}\end{equation*}
Since $X_0$ is path-connected, it follows from Lemma \ref{projection:properties} 
and \eqref{well-defined:Psi} that
\begin{equation}\label{all:iterates:connected}
\tr_\rho(\mc{R}_{h,\rho}^F(k))\in\bdc_\rho\quad\forall\,k\in[0,n].
\end{equation}

\begin{figure}[t]
\centering
\begin{subfigure}[t]{0.3\textwidth}\centering
$\begin{tikzcd}
&\bd_{\rho}
  \arrow[dashed]{r}{\partial R_{\rho}^{\rho'}}
&\bd_{\rho'}\\ 
&S_{\rho}
  \arrow{u}{\tr_{\rho}}
  \arrow{r}{R_{\rho}^{\rho'}} 
&S_{\rho'} \arrow{u}[swap]{\tr_{\rho'}}
\end{tikzcd}$
\end{subfigure}%
\hfill
\begin{subfigure}[t]{0.3\textwidth}\centering
$\begin{tikzcd}
\bdc_\rho
  \arrow[dashed]{r}{\Psi_{h,\rho}} 
&\bdc_\rho\\ 
C_\rho 
  \arrow{u}{\tr_\rho} 
  \arrow{r}{\Phi^{\alpha(h,\rho)}_{h,\rho}}   
&C_\rho
  \arrow{u}[swap]{\tr_\rho}
\end{tikzcd}$%
\end{subfigure}%
\hfill
\begin{subfigure}[t]{0.3\textwidth}\centering
$\begin{tikzcd}
\bdc_\rho
  \arrow[dashed]{r}{\Psi_{h,\rho,\rho'}} 
&\bdc_{\rho'}\\ 
C_\rho 
  \arrow{u}{\tr_\rho} 
  \arrow{r}{\Phi^{\alpha(h,\rho)}_{h,\rho,\rho'}}   
&C_{\rho'}
  \arrow{u}[swap]{\tr_{\rho'}}
\end{tikzcd}$%
\end{subfigure}
\caption{Mathematically equivalent computations carried out 
either with entire sets (bottom row) or with their boundaries (top row).
\label{fig:diagrams}}
\end{figure}

The following theorem is the main result from \cite{Rieger:boundary},
which allows us to evaluate $\Psi_{h,\rho}(D)$ directly, i.e.\
without knowledge of $\tr_\rho^{-1}(D)$. 
Statement \eqref{all:iterates:connected} is essential for its proof,
which relies on topological arguments.

\begin{theorem}\label{thm:boundary:tracking}
Let $h\in(0,\frac{1}{4L}]$ and $\rho>0$, denote $\alpha=\alpha(h,\rho)$ 
from \eqref{alpha} and
\begin{equation}\label{kappa}
\kappa=\kappa(h,\rho):=2(1+Lh)(\tfrac{\alpha}{1-Lh}+\rho),
\end{equation}
and let $D=(D_0,D_1)\in\bdc_\rho$.
Furthermore, consider the set-valued mappings 
$\phi_h^{\partial}:\R^d\to\mc{K}(\R^d)$ 
and $\Phi_{h,\rho}^{\partial,\alpha}:\R^d\to S_\rho^+$
given by
\[\phi_h^\partial(x):=\partial\phi_h(x)
\quad\text{and}\quad
\Phi_{h,\rho}^{\partial,\alpha}(x):=\pi_{\rho}^{\alpha}(\phi_h^\partial(x)),\]
and the sets
\begin{equation}\label{easy:step}
D_{-1}:=\partial_\rho^{-1}(\tr^{-1}(D))
\quad\text{and}\quad 
D_2:=\partial_\rho^{2}(\tr^{-1}(D)).
\end{equation}
Then we can form the auxiliary collections 
\begin{equation}\label{eq:SDefns}\begin{gathered}
W_0^0:= \Phi^{\partial,\alpha+\kappa}_{h,\rho}(D_0)\cap\Phi^\alpha_{h,\rho}(D_0),\quad
W_1=\Phi^{\partial,\alpha}_{h,\rho}(D_1\cup D_2),\\
W_0^{-1}:=\Phi^{\partial,\alpha}_{h,\rho}(D_{-1}),
\end{gathered}\end{equation}
and we can express $\Psi_{h,\rho}(D)=\tilde{D}$
with the pair $\tilde{D}=(\tilde{D}_0,\tilde{D}_1)$ given by
\begin{align*}
\tilde{D}_1&:=\{x\in W_1: \dist(x,W_0^0\cup W_0^{-1})=\rho\},\\
\tilde{D}_0&:=\{x\in W_0^0\cup W_0^{-1}: \dist(x,\tilde{D}_1)=\rho\}
\end{align*}
in terms of $D_{-1},\ D_0,\ D_1$, and $D_2$.
\end{theorem}

Theorem \ref{thm:boundary:tracking} is the theoretical foundation for Algorithm 
\ref{Alg:BdryEulrStep}, which is displayed on page \pageref{Alg:BdryEulrStep}.
Note that it is easy and inexpensive to compute the sets in \eqref{easy:step}.
Figure \ref{fig:SigmaExplain} provides a sketch of the sets used to calculate 
\eqref{eq:SDefns}.

\begin{figure}
    \centering
    \begin{subfigure}[t]{0.3\textwidth}
        \centering
        \begin{tikzpicture}
        \draw[->] (0,0)--(3,0) node[right]{};
        \draw[->] (0,0)--(0,2) node[above]{};
        \draw[gray, thick, dashed, fill=gray, fill opacity=0.4](1.5,1) ellipse[x radius=0.85, y radius = 0.65];
    \end{tikzpicture}
        \caption{$\phi_h(x)$ (gray), and $\phi_h^\partial(x)$ (dashed)}
    \end{subfigure}
    ~
    \begin{subfigure}[t]{0.3\textwidth}
    \centering
    \begin{tikzpicture}
        \draw[->] (0,0)--(3,0) node[right]{};
        \draw[->] (0,0)--(0,2) node[above]{};
        \draw[gray, thick, densely dashed](1.5,1) ellipse[x radius=0.85, y radius=0.65] ;
        \draw[red, thick](1.5,1) ellipse[x radius=0.85+0.15, y radius = 0.65+0.15];
        \draw[red, thick](1.5,1) ellipse[x radius=0.85-2*0.15, y radius = 0.65-2*0.15];
        \fill[red, fill opacity=0.1, even odd rule] (1.5,1) ellipse[x radius=0.85+0.15, y radius = 0.65+0.15] ellipse[x radius=0.85-2*0.15, y radius = 0.65-2*0.15];

        \draw[black, thick, |-|]   (1.5+0.85-2*0.15,1)--(1.5+0.85+0.15,1) node[right]{$2\alpha+\kappa$};
        
    \end{tikzpicture}
    \caption{Blowup required for $x\in D_0$}        
    \end{subfigure}
    ~
    \begin{subfigure}[t]{0.3\textwidth}
        \centering
        \begin{tikzpicture}
        \draw[->] (0,0)--(3,0) node[right]{};
        \draw[->] (0,0)--(0,2) node[above]{};
        \draw[gray, thick, densely dashed](1.5,1) ellipse[x radius=0.85, y radius=0.65] ;
        \draw[red, thick](1.5,1) ellipse[x radius=0.85+0.15, y radius = 0.65+0.15];
        \draw[red, thick](1.5,1) ellipse[x radius=0.85-0.15, y radius = 0.65-0.15];
        \fill[red, fill opacity=0.1, even odd rule] (1.5,1) ellipse[x radius=0.85+0.15, y radius = 0.65+0.15] ellipse[x radius=0.85-0.15, y radius = 0.65+-0.15];

        \draw[black, thick, |-|]   (1.5+0.85-0.15,1)--(1.5+0.85+0.15,1) node[right]{$2\alpha$};
        
        \end{tikzpicture}
        \caption{Blowup required for $x\in D_{k}$, $k\in\{-1,1,2\}$}
    \end{subfigure}
    
    \caption{Sets used in the calculation of $W_0^0,\ W_0^{-1}$ and $W_1$ in equations \eqref{eq:SDefns} and \eqref{where:sigma:comes:from}.}
    \label{fig:SigmaExplain}
\end{figure}
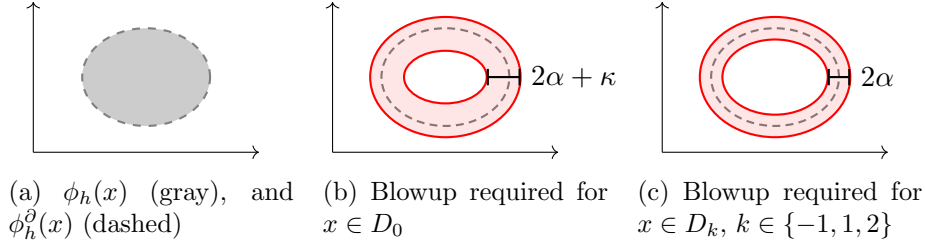

\subsection{An Euler scheme for nonuniform grids}

We present a numerical method that generalizes the Euler scheme
from Section \ref{sec:Euler:scheme} to nonuniform temporal and spatial discretization
parameters
\[{\bm{\delta}}:=(\bm{h},\bm{t},\bm{\rho})\in\R^n_{>0}\times\R_+^{0,n}\times\R_{>0}^{0,n},\]
and is specifically designed to support the boundary tracking algorithm to be 
proposed in Section \ref{sec:bdry:nonuniform}.
The vector $\bm{h}=(h_1,\ldots,h_n)$ contains stepsizes with $\sum_{k=1}^nh_k=T$,
the vector $\bm{t}=(t_0,\ldots,t_n)=\sum_+\bm{h}$ stores the temporal nodes,
and the vector $\bm{\rho}=(\rho_0,\ldots,\rho_n)$ stores the mesh sizes of the 
spatial grids $\Delta_{\rho_0},\ldots,\Delta_{\rho_n}$ associated with the
temporal nodes $t_0,\ldots,t_n$.

\medskip

The following mapping is to be interpreted as a single Euler step with stepsize $h>0$ 
and blow-up $\alpha>0$ from the grid $\Delta_\rho$ to the grid $\Delta_{\rho'}$.
The reason for its particular construction will become apparent in Theorem 
\ref{thm:nonuniformbdryworks}.

\begin{definition}\label{def:NonunifBdryEuler}
Let $\alpha,h,\rho,\rho'>0$ satisfy $\rho'/\rho\in(0,1)\cup\N_1$ as well as
$\alpha\ge\min(\rho,\rho')/2$.
We define the set-valued function $\Phi_{h,\rho,\rho'}^\alpha:\R^d\to S_{\rho'}$ by 
\[\Phi_{h,\rho,\rho'}^\alpha(x)
:=\begin{cases}
R_{\rho}^{\rho'}(\Phi_{h,\rho}^{\alpha}(x)),&\rho<\rho',\\
\Phi_{h,\rho'}^{\alpha}(x),&\rho\ge\rho'
\end{cases}\]
with $\Phi_{h,\rho}^\alpha$ from Definition \ref{def:Euler}.
\end{definition}

Our next result generalizes Proposition \ref{Prop:SetsAreConnected}.

\begin{corollary}\label{Phi:stays:in:C}
Let $h,\rho,\rho'>0$ with $\rho'/\rho\in(0,1)\cup\N_1$,
and let $\alpha(h,\rho)$ be as in Proposition \ref{Prop:SetsAreConnected}. 
Then for any $M\in C_\rho$, we have $\Phi_{h,\rho,\rho'}^{\alpha(h,\rho)}\in C_{\rho'}.$
\end{corollary}

\begin{proof}
When $\rho<\rho'$, then $\Phi_{h,\rho}^{\alpha(h,\rho)}(x)\in C_\rho$
by Proposition \ref{Prop:SetsAreConnected}, and by statement 
\eqref{R:connected}, we have 
$R_{\rho}^{\rho'}(\Phi_{h,\rho}^{\alpha(h,\rho)}(x))\in C_{\rho'}$.
If $\rho\ge\rho'$, then $\alpha(h,\rho)\ge\alpha(h,\rho')$,
and hence $\Phi_{h,\rho'}^{\alpha(h,\rho)}(x)\in C_{\rho'}$
directly by Proposition \ref{Prop:SetsAreConnected}.
\end{proof}

We generalize Definition \ref{reachable:sets:uniform} to nonuniform discretizations.

\begin{definition}\label{def:ApproxReach}
Let $n\in\N_1$, and let
${\bm{\delta}}=(\bm{h},\bm{t},\bm{\rho})\in\R^n_{>0}\times\R_+^{0,n}\times\R_{>0}^{0,n}$
be such that $\sum_{k=1}^nh_k=T$ as well as $\bm{t}=\Sigma_+\bm{h}$ and 
$\rho_{k+1}/\rho_{k}\in(0,1)\cup\N_1$ 
for all $k\in[0,n-1]$. 
Then the set of all Euler approximations to inclusion \eqref{eq:ODI} is
\[\mc{S}_{\bm{\delta}}^F(T)
:=\{(y_k)_{k=0}^n: 
y_0\in \pi_{\rho_0}^{\rho_0/2}(X_0),\ 
y_{k}\in \Phi_{h_k,\rho_{k-1},\rho_k}^{\alpha(h_k,\rho_{k-1})}(y_{k-1})\ 
\forall\, k\in[1,n]\},\]
and the discrete approximation to $\mc{R}^F(t_k)$ from Definition \ref{def:reachable} 
is given by  
\begin{equation}\label{eq:defapproxreach}
\mc{R}_{\bm{\delta}}^F(k):=\{y_k,y\in\mc{S}_{\bm{\delta}}^F(T)\},\quad k\in[0,n].
\end{equation}
\end{definition}

The proof of the following error bound is very similar to proofs in the literature,
see e.g.\ \cite{Beyn:2007} and \cite{Komarov}. 
The additive error terms $\chi_j({\bm{\delta}})$ occur whenever the operator
$R_{\rho_k}^{\rho_{k+1}}$ is invoked, see Definition \ref{def:NonunifBdryEuler}.

\begin{theorem}\label{err:thm:NoE2}
For any $n\in \N_1$, any 
${\bm{\delta}}:=(\bm{h},\bm{t},\bm{\rho})\in\R_{>0}^n\times\R_+^{0,n}\times\R_{>0}^{0,n}$ 
satisfying the requirements of Definition \ref{def:ApproxReach}, and any $k\in[0,n]$, 
the exact reachable sets of inclusion \eqref{eq:ODI} with property \eqref{eq:Pbound}
and the discrete reachable sets from Definition \ref{def:ApproxReach} satisfy
\begin{equation}\label{the:error}\begin{aligned}
&\dist_H(\mc{R}^F(t_k),\mc{R}_{\bm{\delta}}^F(k))\\
&\le\tfrac{\rho_0}{2}e^{Lt_k}+\sum_{j=1}^{k}e^{L(t_{k}-t_j)}
[(e^{Lh_j}-1)(Ph_j+\tfrac{(1+Lh_j)^2}{Lh_j}\tfrac{\rho_{j-1}}{2}
+\chi_j({\bm{\delta}})]
\end{aligned}\end{equation}
for $k\in[0,n]$, where
\[\chi_j({\bm{\delta}}):=\begin{cases}
\tfrac{\rho_j}{2},&\rho_{j-1}<\rho_j,\\0,&\text{otherwise}.
\end{cases}\]
\end{theorem}

As in Section \ref{sec:Euler:scheme}, the discrete reachable sets 
\eqref{eq:defapproxreach} satisfy a recursion 
\begin{equation}\label{nonuniform:recursion}\begin{aligned}
\mc{R}_{\bm{\delta}}^F(0)
&=\pi^{\rho_0/2}_{\rho_0}(X_0),\\
\mc{R}_{\bm{\delta}}^F(k)
&=\Phi^{\alpha(h_k,\rho_{k-1})}_{h_k,\rho_{k-1},\rho_k}(\mc{R}_{\bm{\delta}}^F(k-1))
\quad\forall\,k\in[1,n].
\end{aligned}\end{equation}
By Lemma \ref{projection:properties} and Corollary \ref{Phi:stays:in:C},
we have $\mc{R}_{\bm{\delta}}^F(k)\in C_{\rho_k}$ for all $k\in[0,n]$.

\subsection{Boundary tracking for nonuniform grids}\label{sec:bdry:nonuniform}

In view of Corollary \ref{Phi:stays:in:C}, the mapping
\begin{equation*}
\Psi_{h,\rho,\rho'}(D)
:=\tr_{\rho'}\circ\Phi_{h,\rho,\rho'}^{\alpha(h,\rho)}\circ\tr_{\rho}^{-1}(D)
\quad\forall\,D\in\bdc_\rho
\end{equation*}
is an operator $\Psi_{h,\rho,\rho'}:\bdc_{\rho}\to\bdc_{\rho'}$.
Its action is illustrated in the third diagram in Figure \ref{fig:diagrams}.
It allows us to lift iteration \eqref{nonuniform:recursion}
to the recursion 
\begin{equation}\label{recursion:in:bd}\begin{aligned}
\tr_{\rho_0}(\mc{R}_{\bm{\delta}}^F(0))
&=\tr_{\rho_0}(\pi^{\rho_0/2}_{\rho_0}(X_0)),\\
\tr_{\rho_k}(\mc{R}_{{\bm{\delta}}}^F(k))
&=\Psi_{h_k,\rho_{k-1},\rho_k}
(\tr_{\rho_{k-1}}(\mc{R}_{\bm{\delta}}^F(k-1)))
\quad\forall\,k\in[1,n],
\end{aligned}\end{equation}
which is displayed on page \pageref{Alg:boundary:tracking} as Algorithm 
\ref{Alg:boundary:tracking}.
Again because of Lemma \ref{projection:properties}, we have
$\tr_{\rho_k}(\mc{R}_{{\bm{\delta}}}^F(k))\in\bdc_{\rho_k}$ for all $k\in[0,n]$.
Note that these iterates are the boundary pairs of the sets $\mc{R}_{{\bm{\delta}}}^F(k)$
which satisfy the error bound \eqref{the:error}.

\medskip

To carry out recursion  \eqref{recursion:in:bd} without knowledge of the set 
$\mc{R}_{h,\rho}^F(k-1)$, 
we would ideally like to generalize Theorem \ref{thm:boundary:tracking} 
to nonuniform grids.
An inspection of its proof, given in \cite{Rieger:boundary}, shows that 
the theorem remains valid when $\rho_{k+1}\le\rho_k$.
Elementary counterexamples show, however, that the theorem is false when 
$\rho_{k+1}>\rho_k$.
The construction in Definition \ref{def:NonunifBdryEuler} circumvents
this problem, as it allows the representation \eqref{bdry:eq:nonuniformworks}
below.

\begin{theorem}\label{thm:nonuniformbdryworks} 
Let $h\in(0,\frac{1}{4L}]$, 
and let $\rho,\rho'>0$ satisfy $\rho'/\rho\in(0,1)\cup\N_1$.
Let $\alpha=\alpha(h,\rho)$ and $\kappa=\kappa(h,\rho)$ 
as in \eqref{alpha} and \eqref{kappa}, 
let $\tilde{\rho}:=\min(\rho,\rho')$,
and consider a pair $D=(D_0,D_1)\in\bdc_\rho$.
Denoting
\[D_{-1}:=\partial_\rho^{-1}(\tr^{-1}(D))
\quad\text{and}\quad 
D_2:=\partial_\rho^{2}(\tr^{-1}(D)),\]
we can form the auxiliary collections
\begin{equation}\label{where:sigma:comes:from}\begin{gathered}
W_0^0:= \Phi^{\partial,\alpha+\kappa}_{h,\tilde{\rho}}(D_0)\cap\Phi^\alpha_{h,\tilde{\rho}}(D_0),\quad
W_1:=\Phi^{\partial,\alpha}_{h,\tilde{\rho}}(D_1\cup D_2),\\
W_0^{-1}:=\Phi^{\partial,\alpha}_{h,\tilde{\rho}}(D_{-1})
\end{gathered}\end{equation}
and the pair $\tilde{D}=(\tilde{D}_0,\tilde{D}_1)$ given by
$\tilde{D}_1:=\{x\in W_1:\dist(x,W_0^0\cup W_0^{-1})=\tilde{\rho}\}$
and $\tilde{D}_0:=\{x\in W_0^0\cup W_0^{-1}:\dist(x,\tilde{D}_1)=\tilde{\rho}\}$.
Then we can express
\begin{equation}\label{bdry:eq:nonuniformworks}
\Psi_{h,\rho,\rho'}(D)
=\begin{cases}
\tilde{D},&\rho'\le\rho,\\
\partial R_{\tilde{\rho}}^{\rho'}(\tilde{D}),&\text{else}.
\end{cases}
\end{equation}
in terms of $D_{-1}$, $D_0$, $D_1$ and $D_2$,
using $\partial R_{\tilde{\rho}}^{\rho'}$ from \eqref{def:dR}.
\end{theorem}

\begin{proof}
A proof identical to that of Theorem \ref{thm:boundary:tracking} shows 
$\tilde{D}\in\bdc_{\tilde{\rho}}$ and
\begin{align}\label{bdry:eq:Stildes}
\Psi_{h,\rho,\tilde{\rho}}(D)=\tilde{D}.
\end{align}
If $\rho'\le\rho$, we have $\tilde{\rho}=\rho'$ and \eqref{bdry:eq:Stildes} 
implies \eqref{bdry:eq:nonuniformworks}.
If $\rho'>\rho$, then $\tilde{\rho}=\rho$ and using equation \eqref{def:dR}, 
 Definition \ref{def:NonunifBdryEuler}, and equation 
\eqref{bdry:eq:Stildes}, we verify statement \eqref{bdry:eq:nonuniformworks} 
by the computation
\begin{align*}
&\Psi_{h,\rho,\rho'}(D)
=\tr_{\rho'}\circ\Phi_{h,\rho,\rho'}^\alpha\circ\tr_{\rho}^{-1}(D)\\
&=\tr_{\rho'}\circ R_{\rho}^{\rho'}\circ\Phi_{h,\rho}^\alpha
\circ\tr_{\rho}^{-1}(D)
=\tr_{\rho'}\circ R_{\rho}^{\rho'}\circ\Phi_{h,\rho,\tilde{\rho}}^\alpha
\circ\tr_{\rho}^{-1}(D)\\
&=\partial R_{\rho}^{\rho'}\circ\tr_\rho\circ\Phi_{h,\rho,\tilde{\rho}}
\circ\tr^{-1}(D)
=\partial R_{\rho}^{\rho'}\circ\Psi_{h,\rho,\tilde{\rho}}(D)
=\partial R_{\rho}^{\rho'}(\tilde{D}).
\end{align*}   
\end{proof}

Algorithm \ref{Alg:BdryEulrStep} implements the operator $\Psi_{h,\rho,\rho'}$ 
as proposed in Theorem \ref{thm:nonuniformbdryworks}. 
In addition to an image $\Psi_{h,\rho,\rho'}(D_0,D_1)$, 
it returns the numbers $\hat{c}_R\in\N_{1}$ and $\hat{c}_F\in\N_1$ of grid points,
which can be counted incidentally and are the key ingredients for the 
surrogate volumes \eqref{def:Rvol} and \eqref{def:Fvol}.
An implementation of the operator $\partial R_{\rho}^{\rho'}$,
which occurs in line \ref{R:is:needed}, is available
in \cite[Algorithm 1]{Rieger:grid}.

\section{An abstract subdivision scheme}\label{sec:abstract:scheme}

The primary contribution of this paper is Algorithm \ref{Alg:IterativeMethod}, 
displayed on page \pageref{Alg:IterativeMethod},
which alternates between executing recursion \eqref{recursion:in:bd} and 
refining the discretization based on information gathered in the recursion. 
This section develops an abstract subdivision scheme for the latter aspect,
displayed as Algorithm \ref{Alg:RefineDisc} on page \pageref{Alg:RefineDisc},
which will be applied to concrete cost and error estimators $C$ and $E$ 
in Section \ref{sec:main}.

\medskip

The following discretizations are admissible for the subdivision process. 

\begin{definition}\label{aleph:n}
By $\aleph_n\subset\R_{>0}^n\times\R_+^{n+1}\times\R_{>0}^{n+1}$
we denote the set of all tuples 
${\bm{\delta}}=(\bm{h},\bm{t},\bm{\rho})
\in\R_{>0}^n\times\R_+^{n+1}\times\R_{>0}^{n+1}$
such that
\begin{align}
&\bm{t}=\Sigma_+\bm{h},\quad 
T=t_n\label{eq:htreqs}\\
&\|\bm{h}\|_\infty\le\tfrac{1}{4L},\quad
\|\bm{\rho}\|_\infty\le\tfrac{P}{8L} \label{eq:htbasics}\\
&\tfrac{\rho_j}{\rho_{k}}\in\{4^\ell:\ell\in\Z\}\quad\forall\, j,k\in[0,n],\label{eq:rhofracs}\\
&\rho_j=2LP\left(h_j\right)^2\quad\forall\, j\in[1,n],\label{eq:coupling}
\end{align} 
Additionally, we denote
\begin{equation*}
    \aleph:=\cup_{n\in\N_1}\aleph_n.
\end{equation*}
\end{definition}

We will refine a given discretization by repeated subdivision according to the following rule.

\begin{definition}
Given any $n\in\N_1$ and  ${\bm{\delta}}=(\bm{h},\bm{t},\bm{\rho})\in\aleph_n$, we define the function
\begin{equation}\label{subdivision:rule:full}
\psi[{\bm{\delta}};j]=\left(\psi_h[\bm{h};j], \psi_t[\bm{t};j],\psi_\rho[\bm{\rho};j]\right)
\end{equation}
given by
\begin{equation}\begin{aligned}
\psi_h[\bm{h};j]&:=\begin{cases}
\bm{h},&j=0,\\
(h_1,...,h_{j-1},h_j/2,h_j/2,h_{j+1},...,h_n),&j\in[1,n],
\end{cases}\\
\psi_t[\bm{t};j]&:=\begin{cases}\bm{t},&j=0,\\
(t_0,...,t_{j-1},t_j-h_j/2,t_j,t_{j+1},...,t_n),&j\in[1,n],
\end{cases}\\
\psi_\rho[\bm{\rho};j]&:=\begin{cases}
(\rho_0/4,\rho_1,...,\rho_n),&j=0,\\
(\rho_0,...,\rho_{j-1},\rho_j/4,\rho_j/4,\rho_{j+1},...,\rho_n),&j\in[1,n].
\end{cases}
\end{aligned}\label{subdivision:rules}\end{equation}
\end{definition}

Throughout the rest of this section, we consider two functions
\[C, E:\aleph\to\R_{>0}\]
that assign a number to every admissible discretization,
and the change 
\begin{align*}
&\Delta  C({\bm{\delta}};j):= C(\psi[{\bm{\delta}};j])- C({\bm{\delta}})
\quad\forall\,n\in\N_1,\ \forall\,{\bm{\delta}}\in\aleph_n,\ \forall\,j\in[0,n],\\
&\Delta  E({\bm{\delta}};j):= E(\psi[{\bm{\delta}};j])- E({\bm{\delta}})
\quad\forall\,n\in\N_1,\ \forall\,{\bm{\delta}}\in\aleph_n,\ \forall\,j\in[0,n]
\end{align*}
in these functions caused by the subdivision rule $\psi$.
We assume that these functions have the following properties:

\begin{enumerate}[label=\textbf{P\arabic*}]

\item \label{property:rhotozeroerr} 
There exists a function $\omega_1\in C^0(\R_{>0},\R_{>0})$ such that 
$\lim_{s\to 0}\omega_1(s)=0$ and 
$E({\bm{\delta}})\le\omega_1(\|\bm{\rho}\|_\infty)$ holds for all $n\in\N_1$ 
and ${\bm{\delta}}
\in\aleph_n$.

\item \label{property:ErrDecreases2}
There exists a strictly increasing function $\omega_2\in C^0(\R_+,\R_+)$ such that
$\omega_2(0)=0$ and $\Delta E({\bm{\delta}};j)\le-\omega_2(\rho_j)$ holds for all 
$n\in\N_1$, all ${\bm{\delta}}
\in\aleph_n$ and all $j\in[0,n]$.

\item \label{property:CostIncreasesOnMin}
We have $\Delta C({\bm{\delta}};\arg\min_i\rho_i)>0$ for all $n\in\N_1$ 
and ${\bm{\delta}}\in\aleph_n$. 
    
\item \label{property:CostFracNotZero}
There exists $\omega_4\in C^0(\R_{>0},\R_{>0})$ such that
\[\Delta C({\bm{\delta}};j)\le\omega_4(\rho_j)\Delta C({\bm{\delta}};\arg\min_i\rho_i)\]
holds for all $n\in\N_1$, all ${\bm{\delta}}\in\aleph_n$ and all $j\in[0,n]$.
\end{enumerate}

\begin{remark}
Though it is somewhat counterintuitive, examples show that the differences 
$\Delta C({\bm{\delta}};j)$ can be negative.
For this reason, property \ref{property:CostFracNotZero} does not imply
\ref{property:CostIncreasesOnMin}.
\end{remark}

The following lemma is derived from \eqref{subdivision:rule:full} by a straightforward induction.

\begin{lemma}\label{simple:facts}
Let $N\in\N_0$, and let the sequences 
$({\bm{\delta}}^{(m)})_{m=0}^N$, 
$(k_m)_{m=0}^N$, and $(n_m)_{m=0}^N$
be generated by Algorithm \ref{Alg:RefineDisc} from  
${\bm{\delta}}^{(0)}\in\aleph_{n_0}$.
Then ${\bm{\delta}}^{(m)}\in\aleph_{n_m}$ for all $m\in[0,N]$. 
\end{lemma}

If the spatial discretizations generated by Algorithm \ref{Alg:RefineDisc} 
do not become uniformly arbitrarily fine, then there exists a particular
node at which the discretization eventually remains constant.

\begin{lemma} \label{prop:Undivided}
Assume that Algorithm \ref{Alg:RefineDisc} does not terminate, and that the
sequences 
$({\bm{\delta}}^{(m)})_{m=0}^\infty$, $(k_m)_{m=0}^\infty$, and $(n_m)_{m=0}^\infty$ with
${\bm{\delta}}^{(m)}=(\bm{h}^{(m)},\bm{t}^{(m)},\bm{\rho}^{(m)})$ 
it generates satisfy 
\begin{equation}\label{h:not:to:zero}
\lim_{m\to\infty}\|\bm{\rho}^{(m)}\|_\infty\ne 0.
\end{equation}
Then there exist $m_0\in\N$ and $j_{m_0}\in[0,n_{m_0}]$ such that
for all $m\in[m_0,\infty)$, there exists $j_m\in[0,n_m]$ with
$t_{j_m}^{(m)}=t_{j_{m_0}}^{(m_0)}$ and $\rho_{j_m}^{(m)}=\rho_{j_{m_0}}^{(m_0)}$.
\end{lemma}

\begin{proof}
By statement \eqref{subdivision:rules}, the sequence $(\|\bm{\rho}^{(m)}\|_\infty)_{m=0}^\infty$ 
is a monotone decreasing sequence of positive numbers.
Hence there exists $s\in\R_+$ with $\lim_{m\to\infty}\|\bm{\rho}^{(m)}\|_\infty=s$.
By statement \eqref{h:not:to:zero}, we have $s>0$, and it follows from statement \eqref{eq:rhofracs}
that  there exists $m_0\in\N$ with 
\begin{equation}\label{constant:max}
\|\bm{\rho}^{(m)}\|_\infty=s\quad\forall\,m\in[m_0,\infty).
\end{equation}
Hence the sets
\[S_m:=\{t_j^{(m)}:j\in[0,n_m],\ \rho_j^{(m)}=s\},\quad m\in[m_0,\infty),\] 
are nonempty and closed subsets of the compact set $[0,T]$.
By statements \eqref{subdivision:rules} and \eqref{constant:max}, we also have
$S_{m+1}\subset S_m$ for all $m\in[m_0,\infty)$, so by \cite[Theorem 26.9]{Munkres}, 
we have $\cap_{m\in[m_0,\infty)}S_m\neq\emptyset$, which yields the desired result.
\end{proof}

If Algorithm \ref{Alg:RefineDisc} does not terminate, then the event that the currently 
finest spatial discretization is refined in line \ref{alg:subdivision:step} occurs infinitely often.

\begin{lemma}\label{prop:SubseqOfMinRhos}
If Algorithm \ref{Alg:RefineDisc} does not terminate in finite time and generates 
sequences 
$({\bm{\delta}}^{(m)})_{m=0}^\infty$,
$(k_m)_{m=0}^\infty$, and $(n_m)_{m=0}^\infty$ with
${\bm{\delta}}^{(m)}=(\bm{h}^{(m)},\bm{t}^{(m)},\bm{\rho}^{(m)})$, 
then the set 
$\N':=\{m\in\N_1:\rho_{k_m}^{(m)}=\min_{j\in[0,n_m]}\rho_j^{(m)}\}$ is infinite.
\end{lemma}

\begin{proof}
Assume that $\N'$ is finite, and take $\tilde{m}:=\max(\N')+1$, so that
\begin{equation}\label{conv:eq:rhoequal}
\rho_{k_m}^{(m)}>\min_{j\in[0,n_m]}\rho_j^{(m)}\quad\forall\,m\in[\tilde{m},\infty).
\end{equation}
We assert by induction that this implies
\begin{equation}\label{conv:eq:rhostalls}
\tilde{\rho}:=\min_{j\in[0,n_{\tilde{m}}]}\rho_j^{(\tilde{m})}
=\min_{j\in[0,n_m]}\rho_j^{(m)}
\quad\forall\, m\in[\tilde{m},\infty).
\end{equation}
This is clearly true for $m=\tilde{m}$. 
If \eqref{conv:eq:rhostalls} holds for some $m\in[\tilde{m},\infty)$, 
then by \eqref{eq:rhofracs} and \eqref{conv:eq:rhoequal}, we have 
$\rho_{k_m}^{(m)}\ge 4\min_{j\in[0,n_m]}\rho_j^{(m)}$,
so by statement \eqref{subdivision:rules} we find that \eqref{conv:eq:rhostalls} 
holds with $m+1$ in lieu of $m$.

To derive a contradiction, we introduce the auxiliary function 
\[\beta:[\tilde{m},\infty)\to\R,\quad
\beta(m):=\log_4(\tfrac{\rho_0^{(m)}}{\tilde{\rho}})
+\sum_{j=1}^{n_m}(2^{\log_4(\rho_j^{(m)}/\tilde{\rho})}-1).\]
By \eqref{conv:eq:rhostalls} we have $\rho_j^{(m)}\ge\tilde{\rho}$ 
and hence $\log_4(\rho_j^{(m)}/\tilde{\rho})\ge0$
for all $m\in[\tilde{m},\infty)$ and $j\in[0,n_m]$.
Together with \eqref{eq:rhofracs}, we obtain that
\begin{equation}\label{lemma:aux}
\log_4(\rho_j^{(m)}/\tilde{\rho})\in\N_0
\quad\forall\,m\in[\tilde{m},\infty),\ \forall\,j\in[0,n_m],
\end{equation}
and hence that $\beta(m)\in\N_0$ for all $m\in[\tilde{m},\infty)$.
Using \eqref{subdivision:rules}, it is easy to check that
\[\beta(m+1)=\beta(m)-1,\quad\forall\,m\in[\tilde{m},\infty),\]    
which implies that for $\hat{m}:=\tilde{m}+\beta(\tilde{m})$, 
we have $\beta(\hat{m})=0$.
By \eqref{lemma:aux}, this forces
$\rho^{(\hat{m})}=\tilde{\rho}\mathbbm{1}$.
In particular, we have
\[\rho^{(\hat{m})}_{k_{\hat{m}}}=\tilde{\rho}=\min_{j\in[0,n_{\hat{m}}]}\rho_j^{(\hat{m})},\]
which implies $\hat{m}\in\N'$.
Since $\tilde{m}\le\hat{m}$, this contradicts the definition of $\tilde{m}$.
\end{proof}

Now we prove that Algorithm \ref{Alg:RefineDisc} terminates in finite time.

\begin{theorem}\label{Thm:GeneralCvgc}
If the functions $C$ and $E$ satisfy properties 
\ref{property:rhotozeroerr}--\ref{property:CostFracNotZero},
then, for any 
${\bm{\delta}}^{(0)}\in\aleph_{n_0}$
and $\eps>0$, 
Algorithm \ref{Alg:RefineDisc} terminates after a finite number $m\in\N$ 
of iterations and returns a 
${\bm{\delta}}^{(m)}\in \aleph_{n_m}$
with 
$E({\bm{\delta}}^{(m)})\le\eps$.
\end{theorem}

\begin{proof}
Assume that Algorithm \ref{Alg:RefineDisc} does not terminate.
Then it generates sequences 
$({\bm{\delta}}^{(m)})_{m=0}^\infty$, $(k_m)_{m=0}^\infty$ and $(n_m)_{m=0}^\infty$ with
${\bm{\delta}}^{(m)}=(\bm{h}^{(m)},\bm{t}^{(m)},\bm{\rho}^{(m)})$ for all $m\in\N$
and
\begin{equation}\label{conv:eq:errlb}
E({\bm{\delta}}^{(m)})>\eps\quad\forall\,m\in\N.
\end{equation}
In the following, we show that
\begin{equation}\label{conv:eq:rhotozero}
\lim_{m\to\infty}\|\bm{\rho}^{(m)}\|_\infty=0,
\end{equation}
which, in conjunction with property
\ref{property:rhotozeroerr}, 
contradicts \eqref{conv:eq:errlb}. This, combined with Lemma \ref{simple:facts}, completes the proof.

\medskip

We obtain from property \ref{property:ErrDecreases2} that
\begin{equation}\label{DE:negative}
\Delta E({\bm{\delta}}^{(m)};j)\le-\omega_2(\rho_{j}^{(m)})<0,
\quad\forall\,m\in\N_1,\ \forall\,j\in[0,n_m],
\end{equation} 
and by construction, we have
\begin{equation}\label{wasweissich}
-\sum_{j=0}^{m-1}\Delta E({\bm{\delta}}^{(j)};k_j)
=E({\bm{\delta}}^{(0)})-E({\bm{\delta}}^{(m)})\le E({\bm{\delta}}^{(0)}),\quad\forall\, m\in\N_1.
\end{equation}
Combining \eqref{DE:negative} and \eqref{wasweissich} yields
\begin{equation}\label{conv:eq:tozeroes:a}
\lim_{m\to\infty}\Delta E({\bm{\delta}}^{(m)};k_m)=0.
\end{equation}

If statement \eqref{conv:eq:rhotozero} is false, then by Lemma \ref{prop:Undivided}, 
there exist $m_0\in\N_1$ and $i_{m_0}\in[0,n_{m_0}]$ such that, for every 
$m\in[m_0,\infty)$, there exists $i_m\in[0,n_m]$ with 
\begin{equation}\label{conv:eq:rhoconstan}
\rho_{i_m}^{(m)}=\rho_{i_{m_0}}^{(m_0)}. 
\end{equation}
By Lemma \ref{prop:SubseqOfMinRhos}, the set 
\[\N'=\{m\in\N_1:\rho_{k_m}^{(m)}=\min_{j\in[0,n_m]}\rho_j^{(m)}\}\]
is infinite, and hence $\N'':=[m_0,\infty)\cap\N'$ is infinite as well.
By property \ref{property:CostIncreasesOnMin}, we have
\begin{equation}
\Delta C({\bm{\delta}}^{(m)};k_m)>0 
\quad\forall\, m\in\N''.\label{conv:eq:DeltaCkPositive}
\end{equation}
If 
$\Delta C({\bm{\delta}}^{(m)};j)\le0$ for some $m\in\N''$ and $j\in[0,n_m]$, 
then the condition in line \ref{alg:subdivision:condition} of Algorithm \ref{Alg:RefineDisc} 
is false, and so  
line \ref{alg:subdivision:alternative} 
yields the contradiction 
$\Delta C({\bm{\delta}}^{(m)};k_m)\le0$.
Hence
\[\Delta C({\bm{\delta}}^{(m)};j)>0
\quad\forall\,m\in\N'',\ \forall\,j\in[0,n_m],
\]
so line \ref{alg:subdivision:first:choice} is executed in iteration 
$m$ for all $m\in\N''$, and thus we have
\begin{equation*}
\tfrac{\Delta E({\bm{\delta}}^{(m)};k_m)}
{\Delta C({\bm{\delta}}^{(m)};k_m)}
\le\tfrac{\Delta E({\bm{\delta}}^{(m)};i_m)}
{\Delta C({\bm{\delta}}^{(m)};i_m)}
\quad\forall\,m\in\N''.
\end{equation*}
By \eqref{DE:negative} and \eqref{conv:eq:DeltaCkPositive}, 
by \eqref{DE:negative} and property \ref{property:ErrDecreases2},
by \eqref{conv:eq:rhoconstan}, and by \eqref{conv:eq:tozeroes:a}, 
we have
\begin{align*}
&\lim_{\N''\ni m\to\infty}
\tfrac{\Delta C({\bm{\delta}}^{(m)};k_m)}
{\Delta C({\bm{\delta}}^{(m)};i_m)}
\le\lim_{\N''\ni m\to\infty}
\tfrac{\Delta E({\bm{\delta}}^{(m)};k_m)}
{\Delta E({\bm{\delta}}^{(m)};i_m)}\\
&\le\lim_{\N''\ni m\to\infty}
\tfrac{\Delta E({\bm{\delta}}^{(m)};k_m)}
{-\omega_2(\rho_{i_m}^{(m)})}    
=\lim_{\N''\ni m\to\infty}
\tfrac{\Delta E({\bm{\delta}}^{(m)};k_m)}
{-\omega_2(\rho_{i_{m_0}}^{(m_0)})}=0.
\end{align*}
However, in view of the definition of $\N''$,
property \ref{property:CostFracNotZero} and \eqref{conv:eq:rhoconstan} provide
\begin{equation*}
\tfrac{\Delta C({\bm{\delta}}^{(m)};k_m)}
{\Delta C({\bm{\delta}}^{(m)};i_m)}
\ge\omega_4(\rho^{(m)}_{i_m})^{-1}
=\omega_4(\rho^{(m_0)}_{i_{m_0}})^{-1}>0
\quad\forall\,m\in\N'',
\end{equation*}
contradicting the above statement.
\end{proof}

\section{Efficient discretizations for boundary tracking}\label{sec:main}

We wish to measure the quality of a discretization 
\[{\bm{\delta}}=(\bm{h},\bm{t},\bm{\rho})\in\R^n_{>0}\times\R_+^{0,n}\times\R_{>0}^{0,n}\]
in terms of the a priori error of the corresponding discrete reachable sets 
and the cost of their computation.
This will allow us to select a discretization with a high
benefit-to-cost ratio.
To this end, we split the a priori error bound 
\begin{align}
E({\bm{\delta}})&=\sum_{k=0}^n\mc{E}_k({\bm{\delta}}),\label{eq:errdef}
\end{align}
from Theorem \ref{err:thm:NoE2} for the reachable sets
computed in recursion \eqref{recursion:in:bd} into the contribution
$\mc{E}_0({\bm{\delta}}):=\frac{\rho_0}{2}e^{LT}$ 
from the discretization of the initial set $X_0$ and the contributions
\[\mc{E}_k({\bm{\delta}}):=e^{L(T-t_k)}\Big[(e^{Lh_k}-1)
\big(Ph_k+\tfrac{(1+Lh_k)^2}{Lh_k}\tfrac{\rho_{k-1}}{2}\big)
+\chi_k({\bm{\delta}})\Big],\ k\in(0,n],\]
from the discretization of the dynamics on each subinterval
$(t_{k-1},t_k]$ for $k\in(0,n]$, where we use the notation
\[\chi_k({\bm{\delta}}):=\begin{cases}
\tfrac{\rho_k}{2},&\rho_{k-1}<\rho_k,\\0,&\text{otherwise}.\end{cases}\]
Since there is no straightforward a priori formula 
for the computational cost of recursion \eqref{recursion:in:bd} available, 
we derive a heuristic estimator based on a posteriori knowledge from previously 
performed computations in Section \ref{sec:estimating:cost}.
We verify in Section \ref{sec:verification} that this cost estimator 
and the a priori error estimate \eqref{eq:errdef} satisfy properties P1 to P4,
which means that it is safe to carry out the subdivision scheme from Section 
\ref{sec:abstract:scheme} with these particular inputs.

\subsection{Estimating the computational cost}\label{sec:estimating:cost}

We first collect some preliminaries:
\begin{itemize}
\item [i)] Recall that $d$ is the dimension of the state space, and let
\begin{equation*}
d_x:=\max_{t\in[0,T]}\dim\mc{R}^F(t)
\quad\text{and}\quad
d_u:=\max_{t\in[0,T]}\max_{x\in\mc{R}^F(t)}\dim F(x).
\end{equation*}
Then the numbers
\begin{equation*}
 d_R:=\min(d-1,d_x)
 \quad\text{and}\quad
 d_F:=\min(d-1,d_u)
\end{equation*}
are a good guess for the dimension of the boundary of the reachable sets
and the boundary of the right-hand side $F$, see Figure \ref{fig:bdrydims} 
for an illustration.
From here forward, we assume that $d_R\ge 1$ to rule out
pathological situations. 
\item [ii)] We express the average thickness of the generalized annuli computed in the evaluation of the operator $\Psi_{h,\rho,\rho'}$ 
in line \ref{use:theorem:in:alg:1} of Algorithm \ref{Alg:BdryEulrStep} 
and \eqref{where:sigma:comes:from}, respectively, in terms of the function 
$\sigma:(0,\tfrac{1}{4L}]\times(0,\tfrac{P}{8L}]\to\R_{>0}$ given by
\begin{equation}\label{eq:sigmadef}
\sigma(h,\rho)
:=2\alpha(h,\rho)+\tfrac{1}{4}\kappa(h,\rho)
=\tfrac{\rho}{4}(1+Lh)\left(\tfrac{1+Lh}{1-Lh}+6\right).
\end{equation}
Figure \ref{fig:SigmaExplain} on page \pageref{fig:SigmaExplain} provides a visualization of these annuli. 
\end{itemize}

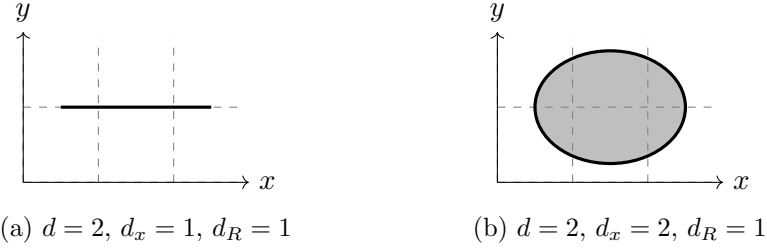
\begin{figure}[t!]\centering
\begin{subfigure}[t]{0.33\textwidth}\centering
\begin{tikzpicture}
    \draw[ color=gray, dashed] (0,0) grid (2.9,1.9);
    \draw[->] (0,0)--(3,0) node[right]{$x$};
    \draw[->] (0,0)--(0,2) node[above]{$y$};
    \draw[-, very thick] (0.5,1)--(2.5,1);
\end{tikzpicture}
\caption{$d=2$, $d_x=1$, $d_R=1$}
\end{subfigure}%
\hspace{2cm}
\begin{subfigure}[t]{0.33\textwidth}\centering
\begin{tikzpicture}
    \draw[color=gray, fill opacity=0.4, dashed] (0,0) grid (2.9,1.9);
    \draw[->] (0,0)--(3,0) node[right]{$x$};
    \draw[->] (0,0)--(0,2) node[above]{$y$};
    \draw[black, very thick, fill=gray, fill opacity=0.5](1.5,1) ellipse (1 and 0.75);
\end{tikzpicture}
\caption{$d=2$, $d_x=2$, $d_R=1$}
\end{subfigure} %
\caption{The dimension $d_R$ of the boundary of a set 
is determined by the dimension $d$ of the state space and the dimension 
$d_x$ of the set.}
\label{fig:bdrydims}
\end{figure}

Now we derive the cost estimator.
The exact cost, in terms of grid points computed, of carrying out recursion 
\eqref{recursion:in:bd} with a given discretization 
${\bm{\delta}}=(\bm{h},\bm{t},\bm{\rho})\in\R^n_{>0}\times\R_+^{0,n}\times\R_{>0}^{0,n}$ is
\begin{equation}\label{def:hat:C}\begin{aligned}
\hat{C}({\bm{\delta}})&:=\sum_{k=0}^{n-1}\hat{\mc{C}}_k({\bm{\delta}}),\\
\hat{\mc{C}}_k({\bm{\delta}})
&:=\sum_{j\in\{-1,1,2\}}
\sum_{x\in\partial_{\rho_k}^j\!\mc{R}_{\bm{\delta}}^F(k)}
\#\Phi_{h_k,\tilde{\rho}_{k+1}}^{\partial,\alpha_{k+1}}(x)\\
&\qquad\qquad+\sum_{x\in\partial_{\rho_{k}}^0\!\mc{R}_{\bm{\delta}}^F(k)}
\#\big(\Phi_{h,\tilde{\rho}_{k+1}}^{\partial,\alpha_{k+1}+\kappa_{k+1}}(x)\cap\Phi_{h,\tilde{\rho}_{k+1}}^{\alpha_{k+1}}(x)\big),
\end{aligned}\end{equation}
where we used the notation
\begin{equation*}
\alpha_k=\alpha(h_k,\rho_{k-1}),\quad 
\kappa_k=\kappa(h_k,\rho_{k-1})
\quad\text{and}\quad
\tilde{\rho}_{k+1}=\min(\rho_k,\rho_{k+1})
\end{equation*}
from \eqref{alpha} and \eqref{kappa}.
It can, however, only be determined after the reachable sets have been computed.

\medskip

For this reason, we pursue an iterative approach, see Algorithm 
\ref{Alg:IterativeMethod}, in which we alternate 
between executing recursion \eqref{recursion:in:bd} and 
refining the discretization based on information gathered in the recursion. 
Assuming that at some point of the iteration, we have determined
\begin{itemize}
\item [i)] a discretization
${\bm{\delta}}^{(\ell)}
\in\aleph_{n_\ell}$,
\item [ii)] the corresponding reachable sets 
$\mc{R}_{{\bm{\delta}}^{(\ell)}}^F(k)$, $k\in[0,n_\ell]$, from \eqref{recursion:in:bd}, and
\item [iii)] the corresponding exact complexities $\hat{\mc{C}}_k({\bm{\delta}}^{(\ell)})$
for $k\in[0,n_\ell-1]$,
\end{itemize}
we proceed as follows:
\begin{itemize}
\item [i)] First we compute the surrogate surface areas
\begin{align}
\hat{v}_{R,k}^{(\ell)}&=\tfrac14(\rho_k^{(\ell)})^{d_R}
\sum_{j=-1}^2\#\partial_{\rho_k^{(\ell)}}^j\mc{R}_{{\bm{\delta}}^{(\ell)}}^F(k),
\quad k\in[0,n_\ell],\label{def:Rvol}\\
\hat{v}_{F,k}^{(\ell)}
&=\tfrac{\min(\rho_{k}^{(\ell)},\rho_{k+1}^{(\ell)})^{d_F+1}}{(h_{k+1}^{(\ell)})^{d_F}\sigma(h_{k+1}^{(\ell)},\rho_k^{(\ell)})}
\tfrac{\hat{\mc{C}}_k({\bm{\delta}}^{(\ell)})}
{\sum_{j=-1}^2\#\partial_{\rho_k^{(\ell)}}^j\mc{R}_{{\bm{\delta}}^{(\ell)}}^F(k)},
\quad k\in[0,n_\ell-1]\label{def:Fvol}
\end{align}
based on the following heuristic: 
Equation \eqref{def:Rvol} estimates the surface area of the reachable sets
$\mc{R}_{{\bm{\delta}}^{(\ell)}}^F(k)$ by counting the average number of grid points 
in the sets $\partial_{\rho_k^{(\ell)}}^{j}\mc{R}_{{\bm{\delta}}^{(\ell)}}^F(k)$
for $j\in[-1,2]$ and multiplying this number by the surface area of one face
of a cube with diameter $\rho_k^{(\ell)}$. 
Equation \eqref{def:Fvol} estimates the average area of $\partial F(x)$
for $x\in\partial\mc{R}_{{\bm{\delta}}^{(\ell)}}^F(k)$ in a similar way, taking into 
account the scaling of $F$ with $h^{(\ell)}_{k+1}$ and the blowup size 
$\sigma(h_{k+1}^{(\ell)},\rho_k^{(\ell)})$ in evaluating 
$\Psi$ with parameters $h_{k+1}^{(\ell)}$, $\rho_{k}^{(\ell)}$ and 
$\min(\rho_{k}^{(\ell)},\rho_{k+1}^{(\ell)})$,
see \eqref{where:sigma:comes:from} and Figure \ref{fig:SigmaExplain}.

\item [ii)] Next, setting $\hat{v}_{F,n_\ell}^{(\ell)}=\hat{v}_{F,n_\ell-1}$, we interpolate the estimates 
$(t^{(\ell)}_j,\hat{v}_{R,j}^{(\ell)})_{j=0}^{n_\ell}$ and 
$(t^{(\ell)}_j,\hat{v}_{F,j}^{(\ell)})_{j=0}^{n_\ell}$  
with piecewise linear splines 
\[v_R^{(\ell)},v_F^{(\ell)}:[0,T]\to\R_{>0}\]
and form their product $v^{(\ell)}:=v^{(\ell)}_Rv^{(\ell)}_F$.

\item [iii)] Finally, for any finer discretization 
${\bm{\delta}}=(\bm{h},\bm{t},\bm{\rho})\in\R^n_{>0}\times\R_+^{0,n}\times\R_{>0}^{0,n}$,
we estimate the number of grid points to be computed in recursion 
\eqref{eq:defapproxreach} by $C(\delta;v^{(\ell)})$, where
\begin{equation}\label{def:C}\begin{aligned}
C({\bm{\delta}};v)
&=\sum_{k=0}^{n-1}{\mc{C}}_k({\bm{\delta}};v),\\
\mc{C}_k({\bm{\delta}};v)
&=\tfrac{4v(t_k)h_{k+1}^{d_F}}
{\rho_k^{d_R}\min(\rho_k,\rho_{k+1})^{d_F+1}}\sigma(h_{k+1},\rho_k),
\end{aligned}\end{equation}
which inverts the heuristic leading to the estimates
\eqref{def:Rvol} and \eqref{def:Fvol}.
\end{itemize}

\begin{algorithm}[p]
\KwIn{\begin{minipage}[t]{0.788\textwidth} 
$h,\rho,\rho'>0$ with $\rho'/\rho\in(0,1)\cup\N_1$,
$(D_0,D_1)\in\bdc_{\rho}$.
\end{minipage}}
\KwOut{\begin{minipage}[t]{0.788\textwidth}
$\Psi_{h,\rho,\rho'}(D_0,D_1)\in\bdc_{\rho'}$,
$\hat{c}_R\in\N_{1}$, $\hat{c}_F\in\N_1$
\end{minipage}}

Compute $D_{-1}$ and $D_2$ from $D_0$ and $D_1$ by neighbor search\;
$\hat{c}_R\gets\sum_{j=-1}^{j=2}\#D_j$\;
$(\alpha,\kappa,\tilde{\rho})\gets
(\alpha(h,\rho),2(1+L h)(\frac{\alpha}{1-Lh}+\rho),\min(\rho,\rho'))$\;
Compute $W_0^0, W_0^{-1}$, and $W_1$ as in Theorem 
\ref{thm:nonuniformbdryworks}\label{use:theorem:in:alg:1}\;
Compute $\tilde{D}_1$ and $\tilde{D}_0$ 
as in Theorem \ref{thm:nonuniformbdryworks}\label{use:theorem:in:alg:2}\;
$\hat{c}_F\gets 
\sum_{j\in\{-1,1,2\}}\sum_{x\in  D_{j}}
\#\Phi^{\partial,\alpha}_{h,\tilde{\rho}}(x)
+\sum_{x\in D_0}
\#(\Phi^{\partial,\alpha+\kappa}_{h,\tilde{\rho}}(x)
\cap\Phi^\alpha_{h,\tilde{\rho}}(x))$\;
\If{$\rho'>\rho$}
{$\Psi_{h,\rho,\rho'}(D_0,D_1)\gets 
\partial R_{\tilde{\rho}}^{\rho'}(\tilde{D}_0,\tilde{D}_1)$\label{R:is:needed}\;}
\Else{$\Psi_{h,\rho,\rho'}(D_0,D_1)\gets (\tilde{D}_0,\tilde{D}_1)$\;}
\caption{A single boundary tracking step.\label{Alg:BdryEulrStep}\\
Evaluates $\Psi_{h,\rho,\rho'}$ and gathers data for cost estimator \eqref{def:C}.}
\end{algorithm}

\begin{algorithm}[p]
\KwIn{discretization ${\bm{\delta}}=(\bm{h},\bm{t},\bm{\rho})\in\aleph$,\\
boundary of initial set $(D_0^0,D_1^0)\in \bdc_{\rho_0}$}
\KwOut{splines $v_R,v_F:[0,T]\to\R_{>0}$,\\
boundaries $(D_0^k,D_1^k)\in\bdc_{\rho_k}$, $k=0,\ldots,n$, of reachable sets}

\For{$k\gets 0$ \KwTo $n-1$}{
$(D_0^{k+1},D_1^{k+1},\hat{c}_{R,{k}},\hat{c}_{F,k})\gets
[\text{Algorithm\ \ref{Alg:BdryEulrStep}}](h_{k+1},\rho_k,\rho_{k+1},D_0^k,D_1^k)$\;
$(\hat{v}_{R,k},\hat{v}_{F,k})
\gets
(\frac14\hat{c}_{R,k}\rho_{k}^{d_R},
\frac{\min(\rho_k,\rho_{k+1})^{d_F+1}}{(h_{k+1})^{d_F}\sigma(h_{k+1},\rho_k)}
\frac{\hat{c}_{F,k}}{\hat{c}_{R,k}})$\;}

$\hat{c}_{R,n}
\gets\sum_{j=-1}^{j=2}\#(\partial_{\rho_n}^j(\tr_{\rho_n}^{-1}(D_0^n,D_1^n)))$\;
$(\hat{v}_{R,n},\hat{v}_{F,n})
\gets(\frac14\hat{c}_{R,n}\rho_n^{d_R},
\hat{v}_{F,n-1})$\; 
$(v_R,v_F)
\gets(\text{[linear spline]}(t,\hat{v}_{R}),
\text{[linear spline]}(t,\hat{v}_{F}))$\;

\caption{Boundary tracking with nonuniform discretization.\\
Implements recursion \eqref{recursion:in:bd} and gathers data 
for estimator \eqref{def:C}.\label{Alg:boundary:tracking}}
\end{algorithm}

\begin{algorithm}[p]
\KwIn{$\eps>0$, discretization ${\bm{\delta}}^{(0)}\in\aleph_{n_0}$, functions $C$ and $E$}
\KwOut{$n_m\in\N_1$, discretization ${\bm{\delta}}^{(m)}\in\aleph_{n_m}$}
$m\gets 0$\;
\While{$E({\bm{\delta}}^{(m)})>\eps$}{
    \If{$\min_j\Delta C({\bm{\delta}}^{(m)};j)>0$}{\label{alg:subdivision:condition}
        $k_m\gets\argmin_j\frac{\Delta E({\bm{\delta}}^{(m)};j)}{\Delta C({\bm{\delta}}^{(m)};j)}$\label{alg:subdivision:first:choice}\;
    }\Else{
        $k_m\gets\argmin_j\Delta C({\bm{\delta}}^{(m)};j)$\;
        \label{alg:subdivision:alternative}
    }
    ${\bm{\delta}}^{(m+1)}
        \gets\psi[{\bm{\delta}}^{(m)};k_m]$\label{alg:subdivision:step}\;
    \textbf{if} $k_m>0$ \textbf{then} {$n_{m+1}\gets n_m+1$} \textbf{else} {$n_{m+1}\gets n_m$}\;
    $m\gets m+1$\;
}
\caption{An abstract subdivision scheme.\\
Finds cost-efficient discretization with guaranteed error bound.}
\label{Alg:RefineDisc}
\end{algorithm}

\begin{algorithm}[p] 
\KwIn{thresholds $\eps_1>\eps_2>\ldots>\eps_{\ell_{\max}}>0$}
\KwOut{$n\in\N_1$, 
discretization ${\bm{\delta}}\in\aleph_n$,\\
boundaries $(D_0^k,D_1^k)\in\bdc_{\rho_k}$, $k=0,\ldots,n$, of reachable sets}

$(n_0,\ell)\gets(\lceil4LT\rceil,0)$\;
${\bm{\delta}}^{(0)}\gets(h^{(0)},t^{(0)},\rho^{(0)})
\gets(\frac{T}{n_0}\mathbbm{1}_{n_0},
\frac{T}{n_0}\Sigma_+\mathbbm{1}_{n_0},
2LP\tfrac{T^2}{n_0^2}\mathbbm{1}_{n_0+1})$\;

\While{$\ell\le\ell_{\max}$}{
$(D_0^0,D_1^0)\gets
\tr_{\rho_0^{(\ell)}}(\pi^{\rho_0^{(\ell)}/2}_{\rho_0^{(\ell)}}(X_0))$\;\label{algl:CalcInitBdry}
$((D_0^k,D_1^k)_{k=0}^{n_\ell},v^{(\ell)}_R,v^{(\ell)}_F)\gets
[\text{Algorithm\ \ref{Alg:boundary:tracking}}]({\bm{\delta}}^{(\ell)},(D_0^0,D_1^0))$\;
\If{$\ell<\ell_{\max}$}
{$(n_{\ell+1},{\bm{\delta}}^{(\ell+1)})
\gets[\text{Algorithm\ \ref{Alg:RefineDisc}}]
(\eps_{\ell+1},{\bm{\delta}}^{(\ell)},C(\,\cdot\,;v^{(\ell)}_Rv^{(\ell)}_F),E(\,\cdot\,))$\;}
$\ell\gets \ell+1$\;}
 
$(n,{\bm{\delta}})\gets(n_{\ell_{\max}},{\bm{\delta}}^{(\ell_{\max})})$\;

\caption{Boundary tracking with iterative refinement.\\
Approximates reachable sets with near-optimal discretization.}
\label{Alg:IterativeMethod}
\end{algorithm}
\subsection{Verifying hypotheses P1 to P4}\label{sec:verification}

In the following, we show that for a given, fixed $v\in C^0([0,T],\R_{>0}),$ the functions $E(\,\cdot\,)$ and $C(\,\cdot\,;v)$ from \eqref{eq:errdef} 
and \eqref{def:C} satisfy properties \ref{property:rhotozeroerr} to
\ref{property:CostFracNotZero}. 

\begin{lemma}
We have
\begin{align}
&1+s\le e^s\quad\forall\,s\in\R,\label{exp:lower:estimate}\\
&e^s\le 1+s+s^2\quad\forall\,s\in[0,3/2],\label{est:exp}\\
&e^\frac{s}{2}(e^\frac{s}{2}-1)\tfrac{(1+s/2)^2}{s}
-(e^{s}-1)\tfrac{(1+s)^2}{2s}
\le-\tfrac{3s}{8}\quad\forall\,s\in(0,1/4].\label{special:est}
\end{align}
\end{lemma}

\begin{proof}
Both \eqref{exp:lower:estimate} and \eqref{est:exp} are well-known.
Using $e^s-1=(e^{\frac{s}{2}}-1)(e^{\frac{s}{2}}+1)$ and 
$e^\frac{s}{2}-1\le\tfrac{s}{2}+\tfrac{s^2}{4}$,
which is \eqref{est:exp} with $s/2$ in lieu of $s$, 
we obtain
\begin{align*}
&e^\frac{s}{2}(e^\frac{s}{2}-1)\tfrac{(1+s/2)^2}{s}-(e^{s}-1)\tfrac{(1+s)^2}{2s}\\
&=\tfrac{1}{2s}(e^\frac{s}{2}-1)[2e^\frac{s}{2}(1+\tfrac{s}{2})^2-(e^\frac{s}{2}+1)(1+s)^2]\\
&=\tfrac{1}{2s}(e^\frac{s}{2}-1)
[(e^\frac{s}{2}-1)-\tfrac{s^2}{2}e^\frac{s}{2}-2s-s^2]\\
&\le\tfrac{1}{2s}(e^\frac{s}{2}-1)
[\tfrac{s}{2}+\tfrac{s^2}{4}-2s-s^2]
\le-\tfrac{1}{2s}(e^\frac{s}{2}-1)\tfrac{3s^2}{4}
\le-\tfrac{3s}{8}.
\end{align*} 
\end{proof}

We collect some statements that follow directly from the definition 
\eqref{eq:sigmadef}.

\begin{lemma}\label{lem:sigma}
The function $\sigma$ is monotone increasing in both arguments.
In particular, the constant $\bar\sigma:=\sigma(\tfrac{1}{4L},\tfrac{P}{8L})$
satisfies
\begin{equation}\label{sigma:bar}
\sigma(h_{k+1},\rho_k)\le\bar\sigma
\quad\forall\, k\in[0,n-1],\ \forall\,(\bm{h},\bm{t},\bm{\rho})\in\aleph_n,\ \forall\,n\in\N_1.
\end{equation}
Furthermore, for all $(h,\rho)\in[0,\tfrac{1}{4L}]\times(0,\tfrac{P}{8L}]$,
we have
\begin{align}
&\tfrac74\rho\le\sigma(h,\rho)=4\sigma(h,\rho/4),\label{eq:sig:rholinear}\\
&\tfrac12\sigma(h,\rho)\le\sigma(h/2,\rho).
\label{eq:sig:hdiff}
\end{align}
\end{lemma}

\begin{proof}

The function $\sigma$ is clearly monotone increasing in $\rho$. A simple computation shows that $\frac{d}{ds}\sigma(s,\rho)>0$ for all $s\in(0,\tfrac{1}{4L}]$ and $\rho\in(0,\frac{P}{8L}]$. With this \eqref{sigma:bar} then follows from \eqref{eq:htbasics}.

\medskip

The inequality in statement \eqref{eq:sig:rholinear} holds since
\begin{align*}
\tfrac74\rho
=\tfrac{\rho}{4}\times 1\times(1+6)
\le\tfrac{\rho}{4}(1+Lh)(\tfrac{1+Lh}{1-Lh}+6)
=\sigma(h,\rho),
\end{align*}
and the identity in statement \eqref{eq:sig:rholinear} follows directly 
from the definition \eqref{eq:sigmadef}.
To show 
\eqref{eq:sig:hdiff} we use monotonicity, \eqref{eq:coupling}, and \eqref{eq:sig:rholinear} with $h/2$ in lieu of $h$ in the computation
\begin{equation*}
    \tfrac{1}{2}\sigma(h,\rho)\le\tfrac{1}{2}\sigma(\tfrac{1}{4L},\rho)=\tfrac{115}{96}\rho<\tfrac{7}{4}\rho\le
\sigma(\tfrac{h}{2},\rho).
\end{equation*}
\end{proof}

Now we check properties \ref{property:rhotozeroerr} to
\ref{property:CostFracNotZero}. 

\begin{lemma}\label{Lem:property:ErrorUpperBd}
The function $E$ from \eqref{eq:errdef} satisfies property \ref{property:rhotozeroerr}
with $\aleph_n$ as in Definition \ref{aleph:n}.
\end{lemma}

\begin{proof}
Let $n\in\N_1$ and ${\bm{\delta}}=(\bm{h},\bm{t},\bm{\rho})\in\aleph_n$. 
We clearly have
\begin{equation}\label{1A}
\frac{\rho_0}{2}e^{LT}\le\frac12e^{LT}\|\rho\|_\infty,
\end{equation}
and using \eqref{eq:coupling}, we find that
\begin{equation}\label{1B}\begin{aligned}
&\sum_{k=1}^ne^{L(T-t_k)}\chi_k({\bm{\delta}})
\le e^{LT}\sum_{k=1}^n\frac{\rho_{k}}{2}
=LPe^{LT}\sum_{k=1}^nh_k^2\\
&\le LPe^{LT}\|h\|_\infty\sum_{k=1}^nh_k
\le LPTe^{LT}\|h\|_\infty
=\tfrac{1}{\sqrt{2}}Te^{LT}\sqrt{LP\|\rho\|_\infty}.
\end{aligned}\end{equation}
It follows from \eqref{eq:htbasics} that $0\le Lh_k\le\frac14$.
Using \eqref{est:exp},
\eqref{exp:lower:estimate} and \eqref{eq:coupling}, 
we find
\begin{equation}\label{1C}\begin{aligned}
&\sum_{k=1}^ne^{L(T-t_k)}(e^{Lh_k}-1)
\Big(Ph_k+\tfrac{(1+Lh_k)^2}{Lh_k}\tfrac{\rho_{k-1}}{2}\Big)\\
&\le e^{LT}\sum_{k=1}^n(1+Lh_k)Lh_k
\Big(Ph_k+\tfrac{(1+Lh_k)^2}{Lh_k}\tfrac{\rho_{k-1}}{2}\Big)\\
&=e^{LT}\sum_{k=1}^n(1+Lh_k)\Big(LPh_k^2+(1+Lh_k)^2\tfrac{\rho_{k-1}}{2}\Big)\\
&\le e^{4LT}\sum_{k=1}^n\Big(LPh_k^2+\tfrac{\rho_{k-1}}{2}\Big)
=e^{4LT}\Big(\sum_{k=1}^nLPh_k^2+\tfrac{\rho_{0}}{2}+\sum_{k=1}^{n-1}LPh_k^2\Big)\\
&\le e^{4LT}\Big(\|\rho\|_\infty+2LP\sum_{k=1}^nh_k^2\Big)
\le e^{4LT}\Big(\|\rho\|_\infty+2LPT\|h\|_\infty\Big)\\
&=e^{4LT}\Big(\|\rho\|_\infty+T\sqrt{2LP\|\rho\|_\infty}\Big).
\end{aligned}\end{equation}
Combining \eqref{1A}, \eqref{1B} and \eqref{1C}, we see that
$E({\bm{\delta}})\le\omega_1(\|\rho\|_\infty)$ with a function
\[\omega_1(s):=\frac12e^{LT}s
+\tfrac{1}{\sqrt{2}}Te^{LT}\sqrt{LPs}
+e^{4LT}\Big(s+T\sqrt{2LPs}\Big),\]
which is continuous on $\R_{>0}$ and satisfies $\lim_{s\to0}\omega_1(s)=0$.
\end{proof}

\begin{lemma}
The functions $\psi$ and $E$ given by \eqref{subdivision:rule:full} and
\eqref{eq:errdef} satisfy property \ref{property:ErrDecreases2}
with $\aleph_n$ as in Definition \ref{aleph:n}.
\end{lemma}

\begin{proof}
In view of inequalities \eqref{result:for:1} and \eqref{result:for:2} 
proved below, we have
\[\Delta E({\bm{\delta}};k)
\le-\min(\tfrac{1}{4}\rho_k,\tfrac{1}{10}\sqrt{L/(2P)}\rho_k^{3/2})
\quad\forall\,k\in[0,n],\] 
and the function 
$\omega_2(s):=\min(\tfrac{1}{4}\rho_k,\tfrac{1}{10}\sqrt{L/(2P)}\rho_k^{3/2})$ is obviously 
continuous and strictly increasing with $\omega_2(0)=0$.
We consider three different cases.

\medskip

\textit{Case 1:} When $k=0$, comparing terms in the sum \eqref{eq:errdef} yields 
\begin{align*}
&\Delta E({\bm{\delta}};0)
=E(\psi[{\bm{\delta}};0])-E({\bm{\delta}})\\
&=\mc{E}_0(\psi[{\bm{\delta}};0])+\mc{E}_1(\psi[{\bm{\delta}};0])
-\mc{E}_0({\bm{\delta}})-\mc{E}_1({\bm{\delta}})\\
&=\tfrac{\rho_0}{8}e^{LT}
+e^{L(T-t_1)}\Big(
(e^{Lh_1}-1)\big(Ph_1+\tfrac{(1+Lh_1)^2}{Lh_1}\tfrac{\rho_0}{8}\big)
+\chi_1(\psi[{\bm{\delta}};0])\Big)\\
&\qquad-\tfrac{\rho_0}{2}e^{LT}
-e^{L(T-t_1)}\Big(
(e^{Lh_1}-1)\big(Ph_1+\tfrac{(1+Lh_1)^2}{Lh_1}\tfrac{\rho_0}{2}\big)
+\chi_1({\bm{\delta}})\Big)\\
&=-\tfrac38\big(e^{LT}+e^{L(T-t_1)}
(e^{Lh_1}-1)\tfrac{(1+Lh_1)^2}{Lh_1}\big)\rho_0
+e^{L(T-t_1)}\big(\chi_1(\psi[{\bm{\delta}};0])-\chi_1({\bm{\delta}})\big).
\end{align*}
By \eqref{eq:rhofracs}, we have
\begin{align*}
&\chi_1(\psi[{\bm{\delta}};0])-\chi_1({\bm{\delta}})\\
&=\begin{cases}\tfrac{\rho_1}{2},&\tfrac{\rho_0}{4}<\rho_1,\\0,&\text{otherwise}\end{cases}
-\begin{cases}\tfrac{\rho_1}{2},&\rho_0<\rho_1,\\0,&\text{otherwise}\end{cases}
=\begin{cases}\tfrac{\rho_1}{2},&\rho_0=\rho_1,\\0,&\text{otherwise}\end{cases}
\le\tfrac{\rho_0}{2},
\end{align*}
and using
\eqref{exp:lower:estimate} we estimate
\begin{equation}\label{result:for:1}\begin{aligned}
&\Delta E({\bm{\delta}};0)
\le e^{L(T-t_1)}\big(-\tfrac38-\tfrac38
(e^{Lh_1}-1)\tfrac{(1+Lh_1)^2}{Lh_1}+\tfrac12\big)\rho_0\\
&\le e^{L(T-t_1)}\big(-\tfrac38-\tfrac38(1+Lh_1)^2+\tfrac12\big)\rho_0
\le-\tfrac14e^{L(T-t_1)}\rho_0
\le-\tfrac14\rho_0.
\end{aligned}\end{equation}

\textit{Case 2:} When $k\in(0,n)$, comparing terms in the sum \eqref{eq:errdef} yields
\begin{equation}\label{eq:overview}\begin{aligned}
&\Delta E({\bm{\delta}};k)
=E(\psi[{\bm{\delta}};k])-E({\bm{\delta}})\\
&=\mc{E}_k(\psi[{\bm{\delta}};k])+\mc{E}_{k+1}(\psi[{\bm{\delta}};k])
+\mc{E}_{k+2}(\psi[{\bm{\delta}};k])
-\mc{E}_k({\bm{\delta}})-\mc{E}_{k+1}({\bm{\delta}}).
\end{aligned}\end{equation}
Substituting the definitions and canceling terms yields
\begin{equation*}\begin{aligned}
&\mc{E}_{k+2}(\psi[{\bm{\delta}};k])-\mc{E}_{k+1}({\bm{\delta}})\\
&=e^{L(T-t_{k+1})}\Big[(e^{Lh_{k+1}}\!-\!1)
\big(Ph_{k+1}+\tfrac{(1+Lh_{k+1})^2}{Lh_{k+1}}\tfrac{\rho_k}{8}\big)
+\chi_{k+2}(\psi[{\bm{\delta}};k])\Big]\\
&\qquad-e^{L(T-t_{k+1})}\Big[(e^{Lh_{k+1}}\!-\!1)
\big(Ph_{k+1}+\tfrac{(1+Lh_{k+1})^2}{Lh_{k+1}}\tfrac{\rho_{k}}{2}\big)
+\chi_{k+1}({\bm{\delta}})\Big]\\
&=e^{L(T-t_{k+1})}\Big[(e^{Lh_{k+1}}-1)
\tfrac{(1+Lh_{k+1})^2}{Lh_{k+1}}(\tfrac{\rho_k}{8}-\tfrac{\rho_k}{2})
+(\chi_{k+2}(\psi[{\bm{\delta}};k])-\chi_{k+1}({\bm{\delta}}))\Big].
\end{aligned}\end{equation*}
Using \eqref{exp:lower:estimate},
we immediately obtain
\[
(e^{Lh_{k+1}}-1)
\tfrac{(1+Lh_{k+1})^2}{Lh_{k+1}}(\tfrac{\rho_k}{8}-\tfrac{\rho_k}{2})
\le-\tfrac38(1+Lh_{k+1})^2\rho_k
\le-\tfrac38\rho_k,
\]
and using \eqref{eq:rhofracs}, we see that
\begin{align*}
&\chi_{k+2}(\psi[{\bm{\delta}};k])-\chi_{k+1}({\bm{\delta}})\\
&=\begin{cases}
\frac{\rho_{k+1}}{2},&\frac{\rho_k}{4}<\rho_{k+1},\\
0,&\text{otherwise}
\end{cases}
-\begin{cases}
\frac{\rho_{k+1}}{2},&\rho_k<\rho_{k+1},\\
0,&\text{otherwise}
\end{cases}
=\begin{cases}
\frac{\rho_{k+1}}{2},&\rho_k=\rho_{k+1},\\
0,&\text{otherwise}
\end{cases}
\le\frac{\rho_k}{2}.
\end{align*}
Combining the above yields
\begin{equation}\label{E:first}
\mc{E}_{k+2}(\psi[{\bm{\delta}};k])-\mc{E}_{k+1}({\bm{\delta}})
\le\tfrac18e^{L(T-t_{k+1})}\rho_k
\le\tfrac{1}{8}e^{L(T-t_k)}\rho_k.
\end{equation}
We estimate the remaining difference
\begin{align*}
&\mc{E}_k(\psi[{\bm{\delta}};k])+\mc{E}_{k+1}(\psi[{\bm{\delta}};k])-\mc{E}_k({\bm{\delta}})\\
&=e^{L(T-t_k+\frac{h_k}{2})}\Big((e^{\frac{Lh_k}{2}}-1)
\big(P\tfrac{h_k}{2}+(1+\tfrac{Lh_k}{2})^2\tfrac{\rho_{k-1}}{Lh_k}\big)
+\chi_k(\psi[{\bm{\delta}};k])\Big)\\
&\qquad+e^{L(T-t_k)}\Big((e^{\frac{Lh_k}{2}}-1)
\big(P\tfrac{h_k}{2}+(1+\tfrac{Lh_k}{2})^2\tfrac{\rho_{k}}{4Lh_k}\big)
+\chi_{k+1}(\psi[{\bm{\delta}};k])\Big)\\
&\qquad-e^{L(T-t_k)}\Big((e^{Lh_k}-1)
\big(Ph_k+(1+Lh_k)^2\tfrac{\rho_{k-1}}{2Lh_k}\big)
+\chi_k({\bm{\delta}})\Big)
\end{align*}
by breaking it into parts.
Since $\chi_{k+1}(\psi[{\bm{\delta}};k])=0$, using \eqref{eq:htbasics}, we obtain
\begin{equation}\label{chi:terms}\begin{aligned}
&e^{L(T-t_k+\frac{h_k}{2})}\chi_k(\psi[{\bm{\delta}};k])
+e^{L(T-t_k)}\chi_{k+1}(\psi[{\bm{\delta}};k])
-e^{L(T-t_k)}\chi_k({\bm{\delta}})\\
&=e^{L(T-t_k)}e^{\frac{Lh_k}{2}}
\begin{cases}
\frac{\rho_k}{8},&\rho_{k-1}<\frac{\rho_k}{4},\\0,&\text{otherwise}
\end{cases}
-e^{L(T-t_k)}
\begin{cases}
\tfrac{\rho_k}{2},&\rho_{k-1}<\rho_k,\\0,&otherwise,
\end{cases}\\
&\resizebox{.9\hsize}{!}{$
\le\begin{cases}
e^{L(T-t_k)}(e^{\frac{Lh_k}{2}}\tfrac{\rho_k}{8}-\tfrac{\rho_k}{2}),
&\rho_{k-1}<\rho_k,\\0,&otherwise
\end{cases}
\le\begin{cases}
-0.35e^{L(T-t_k)}\rho_k,&\rho_{k-1}<\rho_k,\\
0,&otherwise.
\end{cases}$}
\end{aligned}\end{equation}
Using \eqref{est:exp}, we estimate
\begin{equation*}
\begin{aligned}
&e^{L(T-t_k)}(e^{\frac{Lh_k}{2}}-1)(1+\tfrac{Lh_k}{2})^2
\tfrac{\rho_{k}}{4Lh_k}\le e^{L(T-t_k)}(1+\tfrac{Lh_k}{2})^3\tfrac{\rho_{k}}{8}.
\end{aligned}\end{equation*}
Substituting $e^{L(T-t_k+\frac{h_k}{2})}=e^{L(T-t_k)}e^{\frac{Lh_k}{2}}$,
it follows from \eqref{special:est} with $Lh_k$ in lieu of $s$ that
\begin{equation*}
\begin{aligned}
&\resizebox{.9\hsize}{!}{$e^{L(T-t_k+\frac{h_k}{2})}(e^{\frac{Lh_k}{2}}-1)
(1+\tfrac{Lh_k}{2})^2\tfrac{\rho_{k-1}}{Lh_k}
-e^{L(T-t_k)}(e^{Lh_k}-1)
(1+Lh_k)^2\tfrac{\rho_{k-1}}{2Lh_k}$}\\
&\le -\frac{3}{8}Lh_ke^{L(T-t_k)}\rho_{k-1}
\le\begin{cases}
0,&\rho_{k-1}<\rho_k,\\
-\frac{3}{8}Lh_ke^{L(T-t_k)}\rho_{k},&otherwise.
\end{cases}
\end{aligned}\end{equation*}
Finally, algebraic manipulations, \eqref{eq:coupling} and \eqref{exp:lower:estimate}
yield
\begin{equation}\label{P:terms}\begin{aligned}
&e^{L(T-t_k+\frac{h_k}{2})}(e^{\frac{Lh_k}{2}}-1)P\tfrac{h_k}{2}\\
&\qquad+e^{L(T-t_k)}(e^{\frac{Lh_k}{2}}-1)P\tfrac{h_k}{2}-e^{L(T-t_k)}(e^{Lh_k}-1)Ph_k\\
&=-e^{L(T-t_k)}(e^{Lh_k}-1)P\tfrac{h_k}{2}\\
&=-e^{L(T-t_k)}(e^{Lh_k}-1)\tfrac{\rho_k}{4Lh_k}
\le-e^{L(T-t_k)}\tfrac{\rho_k}{4}.
\end{aligned}\end{equation}
Because of \eqref{eq:htbasics} we have $\tfrac38Lh_k<0.35$,
so summing up inequalities \eqref{chi:terms} to \eqref{P:terms} 
and expanding the cubic term yields
\begin{equation}\label{E:second}
\begin{aligned}
&\mc{E}_k(\psi[{\bm{\delta}};k])+\mc{E}_{k+1}(\psi[{\bm{\delta}};k])-\mc{E}_k({\bm{\delta}})\\
&\le e^{L(T-t_k)}\big[(1+\tfrac{Lh_k}{2})^3\tfrac{\rho_{k}}{8}-\tfrac{\rho_{k}}{4}
+\max\{-0.35\rho_k,-\tfrac38Lh_k\rho_k\}\big]\\
&=e^{L(T-t_k)}\rho_k
\left[-\tfrac{1}{8}-\tfrac{3}{16}Lh_k
+\tfrac{3}{32}(Lh_k)^2+\tfrac{1}{64}(Lh_k)^3\right].
\end{aligned}
\end{equation}
In view of \eqref{eq:overview}, combining inequalities \eqref{E:first} and 
\eqref{E:second}, and using \eqref{eq:htbasics} and \eqref{eq:coupling} yields 
\begin{equation}\label{result:for:2}
\begin{aligned}
\Delta E({\bm{\delta}};k)
&\le e^{L(T-t_k)}\rho_kLh_k
[-\tfrac{3}{16}+\tfrac{3}{32}Lh_k+\tfrac{1}{64}(Lh_k)^2]\\
&<-\tfrac{1}{10}e^{L(T-t_k)}\rho_kLh_k\le-\tfrac{1}{10}\sqrt{L/(2P)}\rho_k^{3/2}.
\end{aligned}
\end{equation}

\textit{Case 3:} When $k=n$, then comparing terms in the sum \eqref{eq:errdef} yields
\[
\Delta E({\bm{\delta}};n)
=E(\psi[{\bm{\delta}};n])-E({\bm{\delta}})
=\mc{E}_n(\psi[{\bm{\delta}};n])+\mc{E}_{n+1}(\psi[{\bm{\delta}};n])
-\mc{E}_n({\bm{\delta}}).
\]
Since estimate \eqref{E:second} holds for $k=n$ for the same reasons as above, 
following the same process as in \eqref{result:for:2} yields the same result
for $k=n$.
\end{proof}

\begin{lemma}\label{Lem:property:DeltaCMinRho}
Let $v\in C^0([0,T],\R_{>0})$.
Then there exists $c>0$ such that for any $n\in\N_1$ and any 
${\bm{\delta}}=(\bm{h},\bm{t},\bm{\rho})\in\aleph_n$, the functions $\psi$ and $C$ 
given by \eqref{subdivision:rule:full} and \eqref{def:C} satisfy 
\begin{equation}\label{eq:P3Inequality}
c(\min_i\rho_i)^{-d_R-d_F/2}\le\Delta C({\bm{\delta}};v;\argmin_i\rho_i).
\end{equation}
In particular, the function $C(\,\cdot\,;v):\aleph\to\R_{>0}$ satisfies 
property \ref{property:CostIncreasesOnMin}.
\end{lemma}

\begin{proof}
Defining
\[V_L:=\inf_{t\in[0,T]}v(t)>0,\]
we will verify statement \eqref{eq:P3Inequality} with
\[
c:=\min(7\tfrac{(4^{d_R+d_F}-1)V_L}{(2LP)^{d_F/2}},\tfrac{2^{2d_R+d_F}7V_L}{(2LP)^{d_F/2}}).
\]
Take 
\begin{equation}\label{minimal:k}
k:=\argmin_i\rho_i.
\end{equation}

\textit{Case 1:} When $k=0$, comparing terms in the sum \eqref{def:C}, using \eqref{eq:sig:rholinear} twice, along with \eqref{eq:coupling} and \eqref{minimal:k} yields
\begin{equation*}
\begin{aligned}
&\Delta C({\bm{\delta}};v;0)=C(\psi[{\bm{\delta}};0];v)-C({\bm{\delta}};v)
=\mc{C}_0(\psi[{\bm{\delta}};0];v)-\mc{C}_0({\bm{\delta}};v)\\
&=\tfrac{4v(0)h_1^{d_F}}
{(\frac{\rho_0}{4})^{d_R}\min(\frac{\rho_0}{4},\rho_1)^{d_F+1}}
\sigma(h_1,\tfrac{\rho_0}{4})
-\tfrac{4v(0)h_1^{d_F}}
{\rho_0^{d_R}\min(\rho_0,\rho_1)^{d_F+1}}
\sigma(h_1,\rho_0)\\
&=4v(0)h_1^{d_F}
\left(\tfrac{\sigma(h_1,\frac{\rho_0}{4})}
{(\frac{\rho_0}{4})^{d_R}(\frac{\rho_0}{4})^{d_F+1}}
-\tfrac{\sigma(h_1,\rho_0)}{(\rho_0)^{d_R}(\rho_0)^{d_F+1}}\right)\\
&=\tfrac{4v(0)h_1^{d_F}}{\rho_0^{d_R+d_F+1}}\left(4^{d_R+d_F+1}\sigma(h_1,\tfrac{\rho_0}{4})-\sigma(h_1,\rho_0)\right)\\
&=\tfrac{4v(0)h_1^{d_F}}{\rho_0^{d_R+d_F+1}}(4^{d_R+d_F}-1)\sigma(h_1,\rho_0)
\ge7v(0)h_1^{d_F}(4^{d_R+d_F}-1)\rho_0^{-d_R-d_F}\\
&=7\tfrac{(4^{d_R+d_F}-1)V_L}{(2LP)^{d_F/2}}\rho_1^{d_F/2}\rho_0^{-d_R-d_F}\ge7\tfrac{(4^{d_R+d_F}-1)V_L}{(2LP)^{d_F/2}}\rho_0^{-d_R-d_F/2}.
\end{aligned}\end{equation*}

\textit{Case 2:} 
When $k\in[1,n-2]$, comparing terms in the sum \eqref{def:C} gives
\begin{align*}
&\Delta C({\bm{\delta}};v;k)
=\sum_{i=0}^{n-1}\mc{C}_i(\psi[{\bm{\delta}};k];v)
-\sum_{i=0}^{n-1}\mc{C}_i({\bm{\delta}};v)\notag\\
&=\mc{C}_{k-1}(\psi[{\bm{\delta}};k];v)
+\mc{C}_k(\psi[{\bm{\delta}};k];v)
+\mc{C}_{k+1}(\psi[{\bm{\delta}};k];v)
-\mc{C}_{k-1}({\bm{\delta}};v)-\mc{C}_k({\bm{\delta}};v).
\end{align*}
Using \eqref{minimal:k}, rearranging, and using \eqref{eq:sig:hdiff}, we obtain
\begin{equation}\label{eq:P3:k-1}\begin{aligned}
&C_{k-1}(\psi[{\bm{\delta}};k];v)-C_{k-1}({\bm{\delta}};v)\\
&=\tfrac{4v(t_{k-1})(\frac{h_k}{2})^{d_F}}
{\rho_{k-1}^{d_R}\min(\rho_{k-1},\frac{\rho_k}{4})^{d_F+1}}
\sigma(\tfrac{h_k}{2},\rho_{k-1})
-\tfrac{4v(t_{k-1})h_k^{d_F}}
{\rho_{k-1}^{d_R}\min(\rho_{k-1},\rho_k)^{d_F+1}}
\sigma(h_k,\rho_{k-1})\\   
&=\tfrac{4v(t_{k-1})(\frac{h_k}{2})^{d_F}}{\rho_{k-1}^{d_R}(\frac{\rho_k}{4})^{d_F+1}}
\sigma(\tfrac{h_k}{2},\rho_{k-1})
-\tfrac{4v(t_{k-1})h_k^{d_F}}{\rho_{k-1}^{d_R}\rho_k^{d_F+1}}
\sigma(h_k,\rho_{k-1})\\   
&=\tfrac{4v(t_{k-1})h_k^{d_F}}{\rho_{k-1}^{d_R}\rho_k^{d_F+1}}
\left(2^{d_F+2}\sigma(\tfrac{h_k}{2},\rho_{k-1})-\sigma(h_k,\rho_{k-1})\right)\\
&\ge\tfrac{4v(t_{k-1})h_k^{d_F}}{\rho_{k-1}^{d_R}\rho_k^{d_F+1}}
(2^{d_F+1}-1)\sigma(h_k,\rho_{k-1})>0.
\end{aligned}\end{equation}
Similarly, using \eqref{minimal:k}, rearranging, and using \eqref{eq:sig:rholinear} 
provides
\begin{equation*}
\begin{aligned}
&C_{k+1}(\psi[{\bm{\delta}};k];v)-C_k({\bm{\delta}};v)\\
&=\tfrac{4v(t_k)h_{k+1}^{d_F}}
{(\frac{\rho_k}{4})^{d_R}\min(\frac{\rho_k}{4},\rho_{k+1})^{d_F+1}}
\sigma(h_{k+1},\tfrac{\rho_k}{4})
-\tfrac{4v(t_k)h_{k+1}^{d_F}}{\rho_k^{d_R}\min(\rho_k,\rho_{k+1})^{d_F+1}}
\sigma(h_{k+1},\rho_k)\\
&=\tfrac{4v(t_k)h_{k+1}^{d_F}}{(\frac{\rho_k}{4})^{d_R}(\frac{\rho_k}{4})^{d_F+1}}
\sigma(h_{k+1},\tfrac{\rho_k}{4})
-\tfrac{4v(t_k)h_{k+1}^{d_F}}{\rho_k^{d_R}\rho_k^{d_F+1}}
\sigma(h_{k+1},\rho_k)\\
&=\tfrac{4v(t_k)h_{k+1}^{d_F}}{\rho_k^{d_R+d_F+1}}
\left(4^{d_R+d_F+1}\sigma(h_{k+1},\tfrac{\rho_k}{4})-\sigma(h_{k+1},\rho_k)\right)\\
&=\tfrac{4v(t_k)h_{k+1}^{d_F}}
{\rho_k^{d_R+d_F+1}}(4^{d_R+d_F}-1)\sigma(h_{k+1},\rho_k)>0.
\end{aligned}\end{equation*}
Finally, elementary computations in combination with \eqref{eq:sig:rholinear}
and \eqref{eq:coupling} yield
\begin{equation}\label{eq:P3:k}\begin{aligned}
&C_k(\psi[{\bm{\delta}};k];v)
=\tfrac{4v(t_k-\frac{h_k}{2})(\frac{h_k}{2})^{d_F}}
{(\frac{\rho_k}{4})^{d_R}\min(\frac{\rho_k}{4},\frac{\rho_k}{4})^{d_F+1}}
\sigma(\tfrac{h_k}{2},\tfrac{\rho_k}{4})\\
&=2^{2d_R+d_F+4}
\tfrac{v(t_k-\frac{h_k}{2})h_k^{d_F}}{\rho_k^{d_R+d_F+1}}
\sigma(\tfrac{h_k}{2},\tfrac{\rho_k}{4})\\
&\ge2^{2d_R+d_F}
\tfrac{7v(t_k-\frac{h_k}{2})h_k^{d_F}}{\rho_k^{d_R+d_F}}
\ge\tfrac{2^{2d_R+d_F}7V_L}{(2LP)^{d_F/2}}\rho_k^{-d_R-d_F/2}.
\end{aligned}\end{equation}
Combining equations \eqref{eq:P3:k-1} - \eqref{eq:P3:k} 
implies \eqref{eq:P3Inequality}. 

\textit{Case 3:} If $k=n-1$, we have
\begin{equation*}
\Delta C({\bm{\delta}};v;k)
=\mc{C}_{k-1}(\psi[{\bm{\delta}};k];v)+\mc{C}_k(\psi[{\bm{\delta}};k];v)-\mc{C}_{k-1}({\bm{\delta}};v).
\end{equation*}
Equations \eqref{eq:P3:k-1} and \eqref{eq:P3:k} apply for the same reasons 
as in case 2, and so \eqref{eq:P3Inequality} holds. 
\end{proof}

\begin{lemma}\label{Lem:property:DeltaCFracBound}
Let $v\in C^0([0,T],\R_{>0})$.
Then
the function $C(\,\cdot\,;v)$ from \eqref{def:C} satisfies property \ref{property:CostFracNotZero}, with $\psi$ given by \eqref{subdivision:rule:full}.
\end{lemma}

\begin{proof}
Once again, we denote
\begin{equation}\label{again:k:minimal}
k:=\argmin_{i\in[0,n]}\rho_i.
\end{equation}
Defining
\[V_U:=\sup_{t\in[0,T]}v(t)\in\R_{>0},\]
we aim to show that 
\begin{equation}\label{the:goal}
\Delta C({\bm{\delta}};v;j)\le\rho_k^{-d_R-d_F/2}
\big(
2\tfrac{4^{d_R+d_F+2}\bar\sigma V_UT^{d_F}}{\rho_j^{d_F/2+1}}
+\tfrac{4^{d_R+1}\bar\sigma V_U}{(2LP)^{d_F/2}\rho_j}
+\tfrac{4^{d_F+2}\bar\sigma V_UT^{d_F}}{\rho_j^{d_F/2+1}}
\big)
\end{equation}
holds for all $j\in[0,n]$.
Combining \eqref{the:goal} with Lemma \ref{Lem:property:DeltaCMinRho}, 
which states that there exists $c>0$ such that

\[c\rho_k^{-d_R-d_F/2}\le\Delta C({\bm{\delta}};v;k),\]
then yields property \ref{property:CostFracNotZero} with
\[\omega_4(s):=\tfrac{1}{c}
\big(
2\tfrac{4^{d_R+d_F+2}\bar\sigma V_UT^{d_F}}{s^{d_F/2+1}}
+\tfrac{4^{d_R+1}\bar\sigma V_U}{(2LP)^{d_F/2}s}
+\tfrac{4^{d_F+2}\bar\sigma V_UT^{d_F}}{s^{d_F/2+1}}
\big).\]
We proceed to verify statement \eqref{the:goal}.

\medskip

\textit{Case 1:} 
When $j=0$, comparing terms in the sum \eqref{def:C},
and using \eqref{sigma:bar} and \eqref{again:k:minimal} yields
\begin{equation*}
\begin{aligned}
&\Delta C({\bm{\delta}};v;0)
=\mc{C}_{0}(\psi[{\bm{\delta}};0];v)-\mc{C}_0({\bm{\delta}};v)
\le\mc{C}_{0}(\psi[{\bm{\delta}};0];v)\\
&=\tfrac{4v(0)h_1^{d_F}}
{(\frac{\rho_0}{4})^{d_R}\min(\frac{\rho_0}{4},\rho_1)^{d_F+1}}
\sigma(h_1,\tfrac{\rho_0}{4})
\le\tfrac{4\bar\sigma V_Uh_1^{d_F}}
{(\frac{\rho_0}{4})^{d_R}\min(\frac{\rho_0}{4},\rho_1)^{d_F+1}}\\
&=\begin{cases}
\tfrac{4^{d_R+d_F+2}\bar\sigma V_Uh_1^{d_F}}{\rho_0^{d_R+d_F+1}},
&\frac{\rho_0}{4}\le\rho_1,\\ 
\tfrac{4^{d_R+1}\bar\sigma V_Uh_1^{d_F}}
{\rho_0^{d_R}\rho_1^{d_F+1}},
&\frac{\rho_0}{4}>\rho_1
\end{cases}
\le\begin{cases}
\tfrac{4^{d_R+d_F+2}\bar\sigma V_UT^{d_F}}{\rho_0^{d_F/2+1}\rho_k^{d_R+d_F/2}},
&\frac{\rho_0}{4}\le\rho_1,\\ 
\tfrac{4^{d_R+1}\bar\sigma V_U}{(2LP)^{d_F/2}\rho_0\rho_k^{d_R+d_F/2}},
&\frac{\rho_0}{4}>\rho_1,
\end{cases}
\end{aligned}\end{equation*}
where the last inequality in the case $\frac{\rho_0}{4}>\rho_1$ is derived from
\eqref{eq:coupling} and \eqref{again:k:minimal} by the computation
\begin{align*}
\tfrac{4^{d_R+1}\bar\sigma V_Uh_1^{d_F}}{\rho_0^{d_R}\rho_1^{d_F+1}}
=\tfrac{4^{d_R+1}\bar\sigma V_U}{(2LP)^{d_F/2}\rho_0^{d_R}\rho_1^{d_F/2+1}}
\le\tfrac{4^{d_R+1}\bar\sigma V_U}{(2LP)^{d_F/2}\rho_0\rho_k^{d_R+d_F/2}}.
\end{align*}

\textit{Case 2:}
When $j\in[1,n-1]$, comparing terms in the sum \eqref{def:C} yields
\begin{align*}
\Delta C({\bm{\delta}};v;j)
&=\mc{C}_{j-1}(\psi[{\bm{\delta}};j];v)+\mc{C}_j(\psi[{\bm{\delta}};j];v)\\
&\quad+\mc{C}_{j+1}(\psi[{\bm{\delta}};j];v)
-\mc{C}_{j-1}({\bm{\delta}};v)-\mc{C}_j({\bm{\delta}};v).
\end{align*}
If $\rho_{j-1}\le\frac{\rho_j}{4}$, then Lemma \ref{lem:sigma} yields 
\begin{equation}\label{eq:P4:j-1:first}\begin{aligned}
&\mc{C}_{j-1}(\psi[{\bm{\delta}};j];v)-\mc{C}_{j-1}({\bm{\delta}};v)\\
&=\tfrac{4v(t_{j-1})(\frac{h_j}{2})^{d_F}}
{\rho_{j-1}^{d_R}\min(\rho_{j-1},\frac{\rho_j}{4})^{d_F+1}}
\sigma(\tfrac{h_j}{2},\rho_{j-1})
-\tfrac{4v(t_{j-1})h_j^{d_F}}
{\rho_{j-1}^{d_R}\min(\rho_{j-1},\rho_j)^{d_F+1}}
\sigma(h_j,\rho_{j-1})\\
&\le\tfrac{4v(t_{j-1})(\frac{h_j}{2})^{d_F}}
{\rho_{j-1}^{d_R+d_F+1}}
\sigma(h_j,\rho_{j-1})
-\tfrac{4v(t_{j-1})h_j^{d_F}}
{\rho_{j-1}^{d_R+d_F+1}}
\sigma(h_j,\rho_{j-1})\le0.
\end{aligned}\end{equation}
If, on the other hand, we have $\rho_{j-1}>\frac{\rho_j}{4}$, then \eqref{sigma:bar}
and \eqref{again:k:minimal} yield
\begin{equation}\label{eq:P4:j-1:second}\begin{aligned}
&\mc{C}_{j-1}(\psi[{\bm{\delta}};j];v)-\mc{C}_{j-1}({\bm{\delta}};v)
\le\mc{C}_{j-1}(\psi[{\bm{\delta}};j];v)\\
&=\tfrac{4v(t_{j-1})(\frac{h_j}{2})^{d_F}}
{\rho_{j-1}^{d_R}\min(\rho_{j-1},\frac{\rho_j}{4})^{d_F+1}}
\sigma(\tfrac{h_j}{2},\rho_{j-1})
\le\tfrac{4^{d_F+2}\bar\sigma V_UT^{d_F}}
{\rho_{j-1}^{d_R}\rho_j^{d_F+1}}
\le\tfrac{4^{d_F+2}\bar\sigma V_UT^{d_F}}
{\rho_j^{d_F/2+1}\rho_k^{d_R+d_F/2}}.
\end{aligned}\end{equation}
Again by \eqref{sigma:bar} and \eqref{again:k:minimal}, we obtain
\begin{equation}\label{eq:P4:j}\begin{aligned}
\mc{C}_{j}(\psi[{\bm{\delta}};j];v)
&=\tfrac{4v(t_{j-1}+\frac{h_j}{2})(\frac{h_j}{2})^{d_F}}
{(\frac{\rho_j}{4})^{d_R}\min(\frac{\rho_j}{4},\frac{\rho_j}{4})^{d_F+1}}
\sigma(\tfrac{h_j}{2},\tfrac{\rho_j}{4})\\
&\le\tfrac{4\bar{\sigma}V_UT^{d_F}}
{(\frac{\rho_j}{4})^{d_R+d_F+1}}
\le\tfrac{4^{d_R+d_F+2}\bar{\sigma}V_UT^{d_F}}
{\rho_j^{d_R+d_F+1}}
\le\tfrac{4^{d_R+d_F+2}\bar{\sigma}V_UT^{d_F}}
{\rho_j^{d_F/2+1}\rho_k^{d_R+d_F/2}}.
\end{aligned}\end{equation}
If $\frac{\rho_j}{4}\le\rho_{j+1}$, then \eqref{sigma:bar} 
and \eqref{again:k:minimal} give
\begin{equation*}
\begin{aligned}
&\mc{C}_{j+1}(\psi[{\bm{\delta}};j];v)-\mc{C}_j({\bm{\delta}};v)\le \mc{C}_{j+1}(\psi[{\bm{\delta}};j];v)\\
&=\tfrac{4v(t_j)h_{j+1}^{d_F}}
{(\frac{\rho_j}{4})^{d_R}\min(\frac{\rho_j}{4},\rho_{j+1})^{d_F+1}}
\sigma(h_{j+1},\tfrac{\rho_j}{4})\\
&\le\tfrac{4^{d_R+d_F+2}\bar{\sigma}V_UT^{d_F}}
{\rho_j^{d_R+d_F+1}}
\le\tfrac{4^{d_R+d_F+2}\bar{\sigma}V_UT^{d_F}}
{\rho_j^{d_F/2+1}\rho_k^{d_R+d_F/2}}.
\end{aligned}\end{equation*}
If, on the other hand, we have $\frac{\rho_j}{4}>\rho_{j+1}$, then \eqref{sigma:bar} 
and \eqref{again:k:minimal} yield
\begin{equation*}
\begin{aligned}
&\mc{C}_{j+1}(\psi[{\bm{\delta}};j];v)-\mc{C}_j({\bm{\delta}};v)\le \mc{C}_{j+1}(\psi[{\bm{\delta}};j];v)\\
&=\tfrac{4v(t_j)h_{j+1}^{d_F}}
{(\frac{\rho_j}{4})^{d_R}\min(\frac{\rho_j}{4},\rho_{j+1})^{d_F+1}}
\sigma(h_{j+1},\tfrac{\rho_j}{4})
\le\tfrac{4^{d_R}\bar{\sigma}V_Uh_{j+1}^{d_F}}
{\rho_j^{d_R}\rho_{j+1}^{d_F+1}}\\
&=\tfrac{4^{d_R}\bar{\sigma}V_U}
{(2LP)^{d_F/2}\rho_j^{d_R}\rho_{j+1}^{d_F/2+1}}
\le\tfrac{4^{d_R}\bar{\sigma}V_U}
{(2LP)^{d_F/2}\rho_j\rho_k^{d_R+d_F/2}}.
\end{aligned}\end{equation*}

\textit{Case 3:}
When $j=n$, comparing terms in the sum \eqref{def:C} yields
\begin{align*}
\Delta C({\bm{\delta}};v;n)
&=\mc{C}_{n-1}(\psi[{\bm{\delta}};n];v)+\mc{C}_{n}(\psi[{\bm{\delta}};n];v)
-\mc{C}_{n-1}({\bm{\delta}};v).
\end{align*}
Statements \eqref{eq:P4:j-1:first}, \eqref{eq:P4:j-1:second} and \eqref{eq:P4:j}
hold with $n$ in lieu of $j$ for the same reasons as in case 2.
\end{proof}

We summarize the results of this section.

\begin{theorem}
For any $\ell_{\max}\in\N_1$ and 
$\eps_1>\eps_2>\ldots>\eps_{\ell_{\max}}>0$,
Algorithm \ref{Alg:IterativeMethod} terminates in finite time and 
returns a 
${\bm{\delta}}\in\aleph$ with
$E({\bm{\delta}})\le\eps_{\ell_{\max}}$
and the corresponding reachable sets generated by recursion \eqref{recursion:in:bd}.
\end{theorem}

\begin{proof}
Algorithm \ref{Alg:IterativeMethod} is initialized with the values 
$n_0:=\lceil 4LT\rceil$,
$\bm{h}^{(0)}:=\tfrac{T}{n_0}\mathbbm{1}_{n_0}$,
$\bm{t}^{(0)}:=\Sigma_+\bm{h}^{(0)}$
and
$\bm{\rho}^{(0)}:=(2LPT^2/n_0^2)\mathbbm{1}_{n_0+1}$.
It is easy to verify that we have ${\bm{\delta}}^{(0)}:=(\bm{h}^{(0)},\bm{t}^{(0)},\bm{\rho}^{(0)})\in\aleph_{n_0}$.

Assume that Algorithm \ref{Alg:IterativeMethod} has generated 
$n_\ell\in\N_1$ and ${\bm{\delta}}^{(\ell)}\in\aleph_{n_\ell}$ for some 
$\ell\in[0,\ell_{\max}-1]$. 
By Lemma \ref{projection:properties}, the computation in line \ref{algl:CalcInitBdry} of Algorithm 
\ref{Alg:IterativeMethod} yields a pair $(D_0^0,D_1^0)\in \bdc_{\rho_0}$, 
which is a valid input for Algorithm \ref{Alg:boundary:tracking}.
In view of Algorithm \ref{Alg:boundary:tracking}, the product 
$v^{(\ell)}:=v_R^{(\ell)}v_F^{(\ell)}$ satisfies $\min_{t\in[0,T]}v^{(\ell)}(t)>0$ for all $t\in[0,T]$,
so by Lemmas \ref{Lem:property:ErrorUpperBd} - \ref{Lem:property:DeltaCFracBound} 
the functions $C(\,\cdot\,;v^{(\ell)}):\aleph\to\R_{>0}$ and $E:\aleph\to\R_{>0}$ from \eqref{def:C} and \eqref{eq:errdef}
satisfy properties \ref{property:rhotozeroerr}, \ref{property:ErrDecreases2}, 
\ref{property:CostIncreasesOnMin}, and \ref{property:CostFracNotZero}. 
Hence Theorem \ref{Thm:GeneralCvgc} guarantees that Algorithm \ref {Alg:RefineDisc}
terminates in finite time and computes ${\bm{\delta}}^{(\ell+1)}\in\aleph_{n_{\ell+1}}$
satisfying $E({\bm{\delta}}^{(\ell+1)})\le\eps_{\ell+1}$.

Now the statement of the theorem follows by induction.
\end{proof}

\section{Numerical examples}\label{sec:numerical:examples}

\begin{table}[p]\centering
\begin{tabular}{|l|c|c|c|c|c|} \hline 
$\boldsymbol{\ell}$&\textbf{7}&\textbf{8}&\textbf{9}&\textbf{10}&\textbf{11} \\ 
\hline
$\eps_\ell$  & 3.2E1 & 1.6E1 & 8.0E0 & 4.0E0 & 2.0E0 \\
\hline
Relative error & 5.9E-1 & 2.9E-1 & 1.5E-1 & 7.3E-2 & 3.7E-2\\
\hline
\begin{tabular}[c]{@{}l@{}}
Algorithm \ref{Alg:BdryEulrStep}\\
line \ref{use:theorem:in:alg:1}
\end{tabular} & 1.90E-1 & 1.99E0 & 2.41E1 & 3.70E2 &9.53E3 \\ 
\hline
\begin{tabular}[c]{@{}l@{}}
Algorithm \ref{Alg:BdryEulrStep}\\
line \ref{use:theorem:in:alg:2}
\end{tabular} & 1.28E-1 & 2.00E0 & 1.34E1 & 1.99E2 & 4.71E3 \\ 
\hline
\begin{tabular}[c]{@{}l@{}}
Algorithm \ref{Alg:RefineDisc}\\
\end{tabular} & 2.96E-3 & 1.01E-2 & 5.56E-2 & 2.91E-1 & 1.19E0 \\ 
\hline\end{tabular}

\caption{Runtime in seconds for different components of Algorithm 
\ref{Alg:IterativeMethod} when applied to system \eqref{eqex:Simp7} 
with parameters $L=4$ and $\eps_{\ell_{\max}}=2$. 
\label{Tab:Timebar}}
\end{table}

\begin{figure}[p]
    \centering
    \includegraphics[trim={0 0.25cm 0 0}]{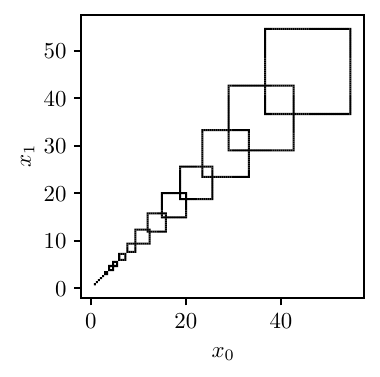}
    \caption{Approximate reachable sets of system \eqref{eqex:Simp7} with $L=4$ generated by Algorithm \ref{Alg:IterativeMethod} with $3.7\%$ relative error. Resolution adjusted for printing purposes. }
    \label{fig:Simp7Reach}
\end{figure}

\begin{figure}[p]
\centering
\includegraphics[trim={0 0.25cm 0 0}]{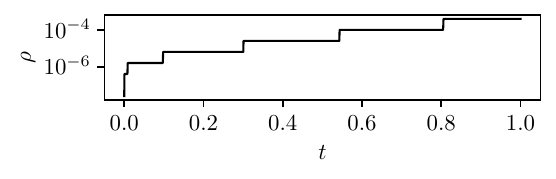}
\caption{Discretization generated by Algorithm \ref{Alg:IterativeMethod} 
to approximate the reachable set of system \ref{eqex:Simp7} with 
$L=4$, and with $3.7\%$ relative error.\label{fig:Simp7rhovals}}
\end{figure}


\begin{sidewaysfigure}[p]
\begin{subfigure}{0.5\textwidth}
\includegraphics[trim={0 0.25cm 0 0}]{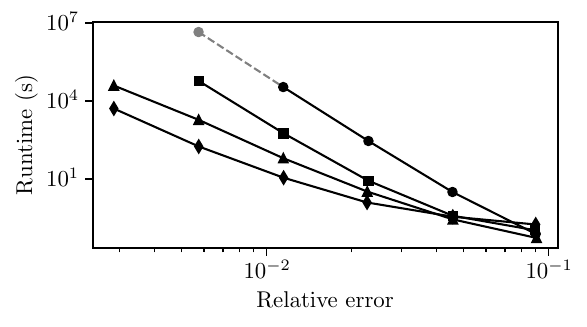}
\caption{$L=1$}
\end{subfigure}%
\begin{subfigure}{0.5\textwidth}
\includegraphics[trim={0 0.25cm 0 0}]{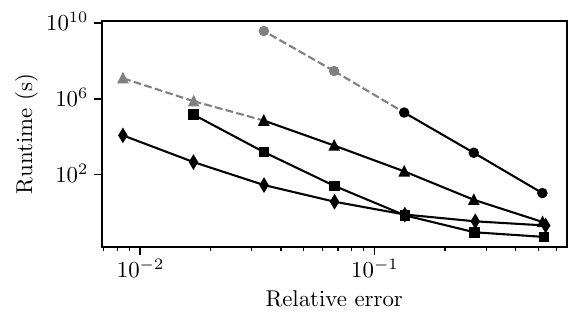}
\caption{$L=2$}
\end{subfigure}

\par\bigskip

\begin{subfigure}{0.5\textwidth}
\includegraphics[trim={0 0.25cm 0 0}]{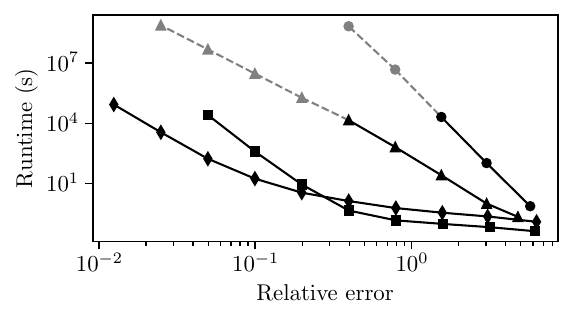}
\caption{$L=3$}
\end{subfigure}%
\begin{subfigure}{0.5\textwidth}
\includegraphics[trim={0 0.25cm 0 0}]{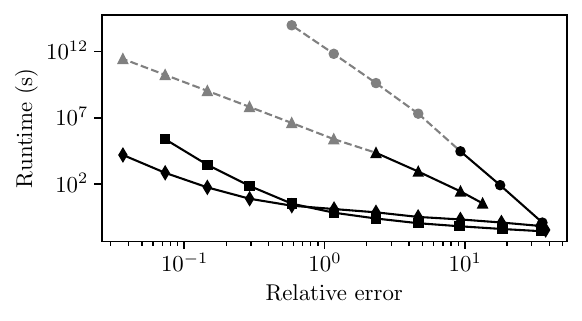}
\caption{$L=4$}
\end{subfigure}

\caption{Algorithms BA (diamonds), BU (triangles), EA (squares) and EU (circles)
applied to system \eqref{eqex:Simp7} with varying $L$.
Black markers represent measurements.
Grey markers represent approximate inferred data.\label{fig:RunTimes}}
    
\end{sidewaysfigure}


\begin{figure}[p]
\centering
\includegraphics[trim={0 0.25cm 0 0}]{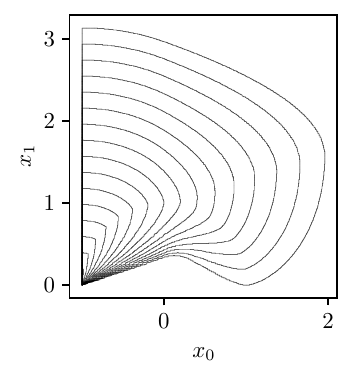}
\caption{Approximate reachable sets of system \eqref{eqex:BP} generated by Algorithm 
\ref{Alg:IterativeMethod} with $8\%$ relative error.
Resolution adjusted for printing purposes.}
\label{fig:BPReach}
\end{figure}

\begin{figure}[p]
\centering
\includegraphics[trim={0 0.25cm 0 0}]{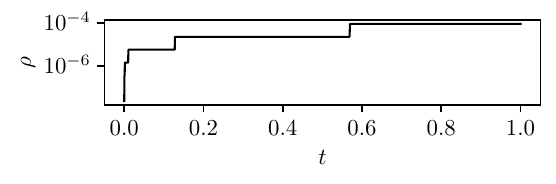}
\caption{Discretization produced by Algorithm \ref{Alg:IterativeMethod} to approximate
the reachable set of system \ref{eqex:BP} with $8\%$ relative error.} 
\label{fig:BPrhovals}
\end{figure}

\begin{figure}[p]
\centering
\includegraphics[trim={0 0.25cm 0 0}]{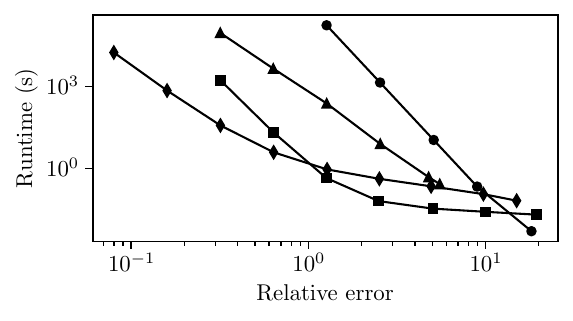}
\caption{Algorithms BA (diamonds), BU (triangles), EA (squares) and EU (circles)
applied to system \eqref{eqex:BP}.
\label{performance:example:2}}
\end{figure}


In this section, we compare Algorithm \ref{Alg:IterativeMethod} with
three other methods for reachable set computation.
We introduce the following abbreviations, where B stands for the \emph{boundary} method, 
E stands for the plain \emph{Euler} scheme, A stands for \emph{adaptive}, and U stands for 
\emph{uniform}:
\begin{itemize}
\item [BA:] Algorithm \ref{Alg:IterativeMethod} from above,
\item [BU:] the boundary tracking algorithm with uniform discretization 
from \cite{Rieger:boundary},
\item [EA:] the Euler scheme with adaptive discretization from \cite{Rieger:2024}, and
\item [EU:] the Euler scheme with uniform discretization from \cite{Beyn:2007}
and \cite{Komarov}.
\end{itemize}
We measure the quality of a numerical solution in terms of the a priori error 
bounds given in the literature for each scheme (bound \eqref{eq:errdef} for Algorithm 
\ref{Alg:IterativeMethod}) and the computational cost in terms of runtime.

\medskip

We first consider a very simple example in which all relevant quantities and the 
reachable sets are explicitly known.

\begin{example}
Let $T=1$ and $L>0$, and consider the differential inclusion in $\R^2$ given by 
\begin{equation}\label{eqex:Simp7}
\dot{x}_i\in[0.9, 1.0]Lx_i,\, i=1,2,\quad x(0)=(1,1)^T.
\end{equation}
The exact reachable sets are 
\begin{equation*}
\mc{R}^F(t)=[\exp(0.9Lt),\exp(Lt)]^2,\quad t\in[0,1],
\end{equation*}
and we have $d_F=d_R=1$ and $P=Le^L$.

\medskip

We first consider Algorithm \eqref{Alg:IterativeMethod} when applied to system \eqref{eqex:Simp7} with $L=4$ and the error tolerance $\eps=2.0$, which converts to a $3.7\%$ relative error. Table \ref{Tab:Timebar} shows how the runtime of Algorithm \ref{Alg:IterativeMethod} is distributed.

\begin{itemize}
\item [i)] Line \ref{use:theorem:in:alg:1} of Algorithm \ref{Alg:BdryEulrStep}
is the part of the algorithm that is captured by the cost $\hat{C}$ from
\eqref{def:hat:C} and hence by the cost estimator $C$ from \eqref{def:C}.

\item [ii)] Since line \ref{use:theorem:in:alg:2} of Algorithm \ref{Alg:BdryEulrStep}
performs a small fixed number of lookup operations for every unique point computed
in line \ref{use:theorem:in:alg:1}, its complexity should be 
a small 
multiple of the complexity of line \ref{use:theorem:in:alg:1}.
Our computations confirm this.
\item [iii)] The computational cost of Algorithm \ref{Alg:RefineDisc} is negligible
compared to that of lines \ref{use:theorem:in:alg:1} and \ref{use:theorem:in:alg:2} 
of Algorithm \ref{Alg:BdryEulrStep}.
\end{itemize}
All in all, the computations confirm that the cost estimator $C$ from \eqref{def:C}
is a sensible measure of the computational complexity of Algorithm \ref{Alg:IterativeMethod}.

Figure \ref{fig:Simp7Reach} shows the resulting approximate reachable set boundaries generated by Algorithm \ref{Alg:IterativeMethod}. Sets shown are taken at times $t\in\{\tfrac{j}{16}:0\le j\le 16\}.$

\medskip

Figure \ref{fig:Simp7rhovals} shows that the adaptive refinement strategy of Algorithm 
\ref{Alg:IterativeMethod} works as intended.
Initially, the reachable set is small, and a very fine temporal and spatial 
discretization can be selected at a low cost.
Since the local errors are propagated exponentially,
choosing a small mesh size near $t=0$ is very desirable.
The opposite arguments hold for larger $t$.

\medskip

To test the performance of Algorithm \ref{Alg:IterativeMethod} (BA) against previous algorithms BU, EA, and EU, we now consider system \eqref{eqex:Simp7} with $L=1,2,3,4$. Figure \ref{fig:RunTimes} shows that in terms of runtime, Algorithm BA from this paper performs best.
Interestingly, the rate at which the runtime increases is very similar for both 
boundary methods, BA and BU, as well as for both Euler methods EA and EU.
\end{example}

The following example is inspired by \cite[Example 4.1]{Riedl:2021}. 

\begin{example}
Let $T=1$, and consider the problem
\begin{equation}\label{eqex:BP}
\dot{x}_1\in[0,1]\pi x_2,\quad\dot{x}_2\in-[0,1]\pi x_1,\quad x(0)=(-1,0)\in\R^2
\end{equation}
with Lipschitz constant $L=\pi$ and bound $P=\pi^2.$ 

\medskip

The approximate reachable set boundaries generated by Algorithm \ref{Alg:IterativeMethod} 
are shown in Figure \ref{fig:BPReach} at times 
$t\in\{\frac{j}{16}:0\le j\le 16\}$.

\medskip

Figures \ref{fig:BPrhovals} and \ref{performance:example:2} are the equivalent of Figures \ref{fig:Simp7rhovals} and \ref{fig:RunTimes} for system \eqref{eqex:BP}, and show that Algorithm \ref{Alg:IterativeMethod} still exhibits the desired behavior when applied to a less straightforward system. 
\end{example}

\section{Conclusion}

Runge-Kutta methods for computing reachable sets have seen significant advancements 
over the past decades, as illustrated in Figures \ref{fig:RunTimes} and 
\ref{performance:example:2}. 
The Euler scheme with uniform discretization from \cite{Beyn:2007,Komarov} achieved 
meaningful error bounds only for control systems with very small control sets 
or over very short time intervals. 
The boundary tracking algorithm with uniform discretization from
\cite{Rieger:boundary} reduced the computational cost dramatically, but
computations with decent accuracy remained a challenge.

The adaptive Euler scheme from \cite{Rieger:2024} outperforms 
\cite{Rieger:boundary} at low and moderate accuracy, but 
underperforms it asymptotically, because it tracks full-dimensional sets.
Algorithm \ref{Alg:IterativeMethod} presented in this paper integrates techniques 
from both \cite{Rieger:boundary} and \cite{Rieger:2024} to overcome this
issue.
All in all, we see in Figures \ref{fig:RunTimes} and \ref{performance:example:2}
that the computational cost for a prescribed precision has been brought down
by many orders of magnitude.

\medskip

While we have proved that Algorithm \ref{Alg:IterativeMethod} converges, 
we have not been able to prove bounds on its performance 
that reflect its behavior in computational experiments.
For now, this remains an open problem.
In addition, it would be desirable to develop boundary tracking methods
and adaptive refinement techniques based on semi-implicit and fully 
implicit schemes, which are better suited for stiff systems.

\subsection*{Acknowledgement}

This work has partly been supported by an Australian Government Research Training
Program (RTP) Scholarship.

\bibliographystyle{hplain}
\bibliography{main}

\end{document}